\documentclass[reqno]{amsart}

\usepackage{amsmath}
\usepackage{amsthm}
\usepackage{amsfonts}
\usepackage{amssymb}
\usepackage{enumerate}
\usepackage{graphicx}
\usepackage{epstopdf}

\newcommand{\Le}{\mathcal{L}}

\DeclareMathOperator{\tr}{tr}

\DeclareMathOperator{\diag}{diag}
\DeclareMathOperator{\dist}{dist}

\DeclareMathOperator{\supp}{supp}
\DeclareMathOperator{\conv}{Conv}
\DeclareMathOperator{\sgn}{sgn}

\newcommand{\Su}{\operatorname{S}}

\newcommand{\M}{\operatorname{M}}
\newcommand{\N}{\operatorname{N}}

\newcommand{\Prob}{\mathbb{P}}
\newcommand{\E}{\mathbb{E}}
\newcommand{\C}{\mathbb{C}}
\renewcommand{\P}{\mathbb{P}}
\renewcommand\Re{\operatorname{Re}}
\renewcommand\Im{\operatorname{Im}}
\newcommand{\eps}{\varepsilon}

\def\R{\mathbb{R}}

\newcommand{\In}{\mathrm{in}}
\newcommand{\Out}{\mathrm{out}}
\newcommand{\Bj}{\mathbf{1}}

\theoremstyle{plain}
  \newtheorem{theorem}{Theorem}[section]
  
  \newtheorem{lemma}[theorem]{Lemma}
  \newtheorem{corollary}[theorem]{Corollary}
  
  \newtheorem{proposition}[theorem]{Proposition}

\theoremstyle{definition}
  \newtheorem{definition}[theorem]{Definition}
  \newtheorem{example}[theorem]{Example}

\theoremstyle{remark}
  \newtheorem{remark}[theorem]{Remark}

\usepackage{bbm}
\newcommand{\abs}[1]{\left\vert#1\right\vert} 
\newcommand{\set}[1]{\left\{#1\right\}} 
\newcommand{\BB}{\mathcal{B}} 
\newcommand{\D}{\mathbb D}
\newcommand{\ind}{\mathbbm 1} 

\begin{document}
\title[Pairing between zeros and critical points]{Pairing between zeros and critical points of random polynomials with independent roots} 

\author{Sean O'Rourke}
\address{Department of Mathematics, University of Colorado at Boulder, Boulder, CO 80309 }
\email{sean.d.orourke@colorado.edu}
\thanks{S. O'Rourke has been supported in part by NSF grant ECCS-1610003.}

\author{Noah Williams}
\address{Department of Mathematics, University of Colorado at Boulder, Boulder, CO 80309 }
\email{noah.williams@colorado.edu}

\begin{abstract}
Let $p_n$ be a random, degree $n$ polynomial whose roots are chosen independently according to the probability measure $\mu$ on the complex plane.  For a deterministic point $\xi$ lying outside the support of $\mu$, we show that almost surely the polynomial $q_n(z):=p_n(z)(z - \xi)$ has a critical point at distance $O(1/n)$ from $\xi$.  In other words, conditioning the random polynomials $p_n$ to have a root at $\xi$ almost surely forces a critical point near $\xi$.  More generally, we prove an analogous result for the critical points of $q_n(z):=p_n(z)(z - \xi_1)\cdots (z - \xi_k)$, where $\xi_1, \ldots, \xi_k$ are deterministic.  In addition, when $k=o(n)$, we show that the empirical distribution constructed from the critical points of $q_n$ converges to $\mu$ in probability as the degree tends to infinity, extending a recent result of Kabluchko \cite{K}.
\end{abstract}

\maketitle

\section{Introduction}
This article deals with the relationship between zeros and critical points of random polynomials in one complex variable.  Recall that a \emph{critical point} of a polynomial $f$ is a root of its derivative $f'$.  There are many results concerning the location of critical points of polynomials whose roots are known.  One of the most famous examples is the Gauss--Lucas theorem, which offers a geometric connection between the roots of a polynomial and the roots of its derivative.  

\begin{theorem}[Gauss--Lucas; Theorem 6.1 from \cite{M}] \label{thm:gauss}
If $f$ is a non-constant polynomial with complex coefficients, then all zeros of $f'$ belong to the convex hull of the set of zeros of $f$.  
\end{theorem}

There are many refinements of Theorem \ref{thm:gauss}; we refer the reader to \cite{Azero,Bzero,CM,Dzero,Drozero,GRR,Jzero,Mzero,Malzero,Mcon,Pzero,Rzero,RS,Szero,Szero2,Smahler,Tzero} and references therein.  

A probabilistic version of the problem was first studied by Pemantle and Rivin \cite{PR}.  Specifically, Pemantle and Rivin raised the following question.  For a random polynomial $f$, when are the zeros of $f'$ stochastically similar to the zeros of $f$?  Before introducing their results, we fix the following notation.  For a polynomial $f$ of degree $n$, we define the empirical measure constructed from the roots of $f$ as 
$$ \mu_{f} := \frac{1}{n} \sum_{z \in \mathbb{C} : f(z) = 0} \delta_{z}, $$
where each root in the sum is counted with multiplicity and $\delta_z$ is the unit point mass at $z$.  In particular, when $f$ is a random polynomial, $\mu_f$ becomes a random probability measure.  For the critical points of $f$, we introduce the notation
$$ \mu_{f}' := \mu_{f'}. $$
In other words, $\mu_{f}'$ is the empirical measure constructed from the critical points of $f$.  

Let $X_1, X_2,\ldots$ be independent and identically distributed (iid) random variables taking values in $\mathbb{C}$, and let $\mu$ be their common probability distribution.  For each $n \geq 1$, define the polynomial
\begin{equation} \label{eq:pnprod}
	p_n(z) := \prod_{j=1}^n(z-X_j).
\end{equation}
Under the assumption that $\mu$ has finite one-dimensional energy, Pemantle and Rivin \cite{PR} show that $\mu'_{p_n}$ converges weakly to $\mu$ as $n$ tends to infinity.  Let us recall what it means for a sequence of random probability measures to converge weakly. 

\begin{definition}[Weak convergence of random probability measures]
Let $T$ be a topological space (such as $\mathbb{R}$ or $\mathbb{C}$), and let $\mathcal{B}$ be its Borel $\sigma$-field.  Let $(\mu_n)_{n\geq 1}$ be a sequence of random probability measures on $(T,\mathcal{B})$, and let $\mu$ be a probability measure on $(T,\mathcal{B})$.  We say \emph{$\mu_n$ converges weakly to $\mu$ in probability} as $n \to \infty$ (and write $\mu_n \to \mu$ in probability) if for all bounded continuous $f:T \to \mathbb{R}$ and any $\eps > 0$,
$$ \lim_{n \to \infty} \Prob \left( \left| \int f d\mu_n - \int f d\mu \right| > \eps \right) = 0. $$
In other words, $\mu_n \to \mu$ in probability as $n \to \infty$ if and only if $\int f d\mu_n \to \int f d\mu$ in probability for all bounded continuous $f: T \to \mathbb{R}$.  Similarly, we say \emph{$\mu_n$ converges weakly to $\mu$ almost surely} as $n \to \infty$ (and write $\mu_n \to \mu$ almost surely) if for all bounded continuous $f:T \to \mathbb{R}$,
$$ \lim_{n \to \infty} \int f d\mu_n = \int f d\mu $$
almost surely.   
\end{definition}

Kabluchko \cite{K} generalized the results of Pemantle and Rivin to the following.  

\begin{theorem}[Kabluchko; \cite{K}] \label{thm:PMK}
Let $\mu$ be an arbitrary probability measure on $\mathbb{C}$, and let $X_1, X_2, \ldots$ be a sequence of iid random variables with distribution $\mu$.  For each $n \geq 1$, let $p_n$ be the degree $n$ polynomial given in \eqref{eq:pnprod}.  Then $\mu'_{p_n}$ converges weakly to $\mu$ in probability as $n \to \infty$.  
\end{theorem}

Subramanian, in \cite{S}, verified a special case of Theorem \ref{thm:PMK} when $\mu$ is supported on the unit circle in the complex plane.  

Naturally, one may ask whether the assumptions in Theorem \ref{thm:PMK} (such as the roots $X_1, X_2, \ldots$ being independent) can be relaxed.  In \cite{O}, the first author managed to prove a version of Theorem \ref{thm:PMK} for random polynomials with dependent roots provided the roots lie on the unit circle and satisfy a number of technical conditions.  In particular, the results in \cite{O} apply to characteristic polynomials of random unitary matrices and other matrices from the classical compact groups (the eigenvalues of such matrices are known to not be independent). Similar results for characteristic polynomials of nearly Hermitian matrices were studied in \cite[Section 2.5]{OW}.  In \cite{TRR}, Reddy considers polynomials whose zeros are chosen randomly from two deterministic sequences of complex numbers in which the empirical measures for both sequences converge to the same limit.  It is shown that the limiting empirical measure of the zeros and critical points agree for these polynomials, yielding a version of Theorem \ref{thm:PMK} where the randomness can be reduced and independence still remains.  However, as the following example shows, the randomness in Theorem \ref{thm:PMK} cannot be completely eliminated (i.e., the theorem does not always hold for sequences of deterministic polynomials).  
\begin{example} 
Let $p_n(z) := z^n - 1$.  Then the roots of $p_n$ are the $n$-th roots of unity, and so $\mu_{p_n}$ converges weakly to the uniform measure on the unit circle as $n$ tends to infinity.  However, all $n-1$ critical points of $p_n$ are located at the origin.  Hence, $\mu'_{p_n} = \delta_0$ for all $n$.  
\end{example}

\subsection{Asymptotic notation}
We use asymptotic notation (such as $O,o$) under the assumption that $n \to \infty$.  In particular, $X= O(Y)$, $Y = \Omega(X)$, $X \ll Y$, and $Y \gg X$ denote the estimate $|X| \leq C Y$, for some constant $C > 0$ independent of $n$ and for all $n \geq C$.  If we need the constant $C$ to depend on another constant, e.g. $C = C_k$, we indicate this with subscripts, e.g. $X = O_{k}(Y)$, $Y = \Omega_k(X)$, $X \ll_k Y$, and $Y\gg_k X$.  We write $X = o(Y)$ if $|X| \leq c(n) Y$ for some $c(n)$ that goes to zero as $n \to \infty$.  Specifically, $o(1)$ denotes a term which tends to zero as $n \to \infty$.

\section{Main results}

To introduce our results, we first consider a special case of the polynomial $p_n$, defined in \eqref{eq:pnprod}, when $\mu$ is the uniform probability measure on the unit circle centered at the origin.  In this case, Theorem \ref{thm:PMK} implies that $\mu_{p_n}'$ converges weakly in probability to $\mu$ as $n \to \infty$.  A numerical simulation of this result is shown in Figure \ref{fig:IIDCirc}; as can be seen, all critical points of $p_n'$ lie very close to the unit circle.  On the other hand, if we consider the polynomial $(z-\xi)p_n(z)$ for some deterministic point $\xi$ outside the unit circle, we see in Figure \ref{fig:IIDCircDetOut} that one of the critical points leaves the unit disk and lies very close to $\xi$.  However, the remaining critical points still lie close to the unit circle.  


\begin{figure}
\includegraphics[width =.8\columnwidth]{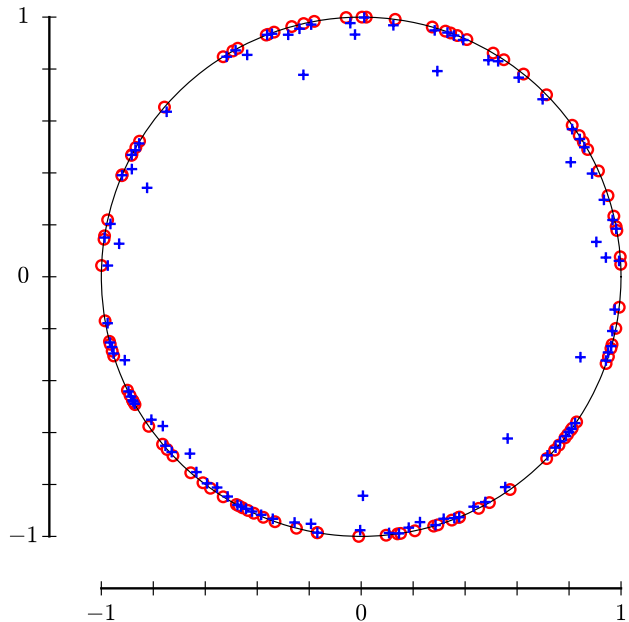}
\caption{The roots (red circles) and critical points (blue crosses) of a random, degree 100 polynomial, where all $100$ roots are chosen independently and uniformly on the unit circle (black curve).}
\label{fig:IIDCirc}
\end{figure}

\begin{figure}
\includegraphics[width =.8\columnwidth]{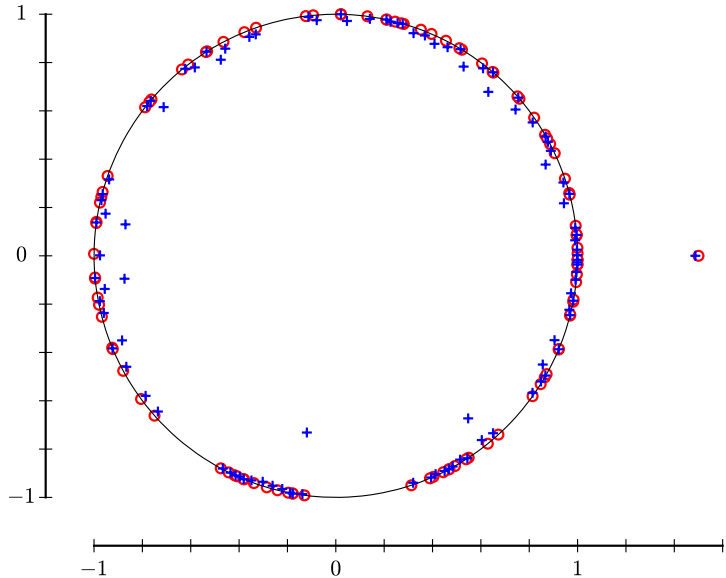}
\caption{The roots (red circles) and critical points (blue crosses) of a random, degree $101$ polynomial, where $100$ roots are chosen independently and uniformly on the unit circle (black curve), and one root takes the deterministic value $\xi = 1.5$.}
\label{fig:IIDCircDetOut}
\end{figure}

\begin{figure}
\includegraphics[width =.8\columnwidth]{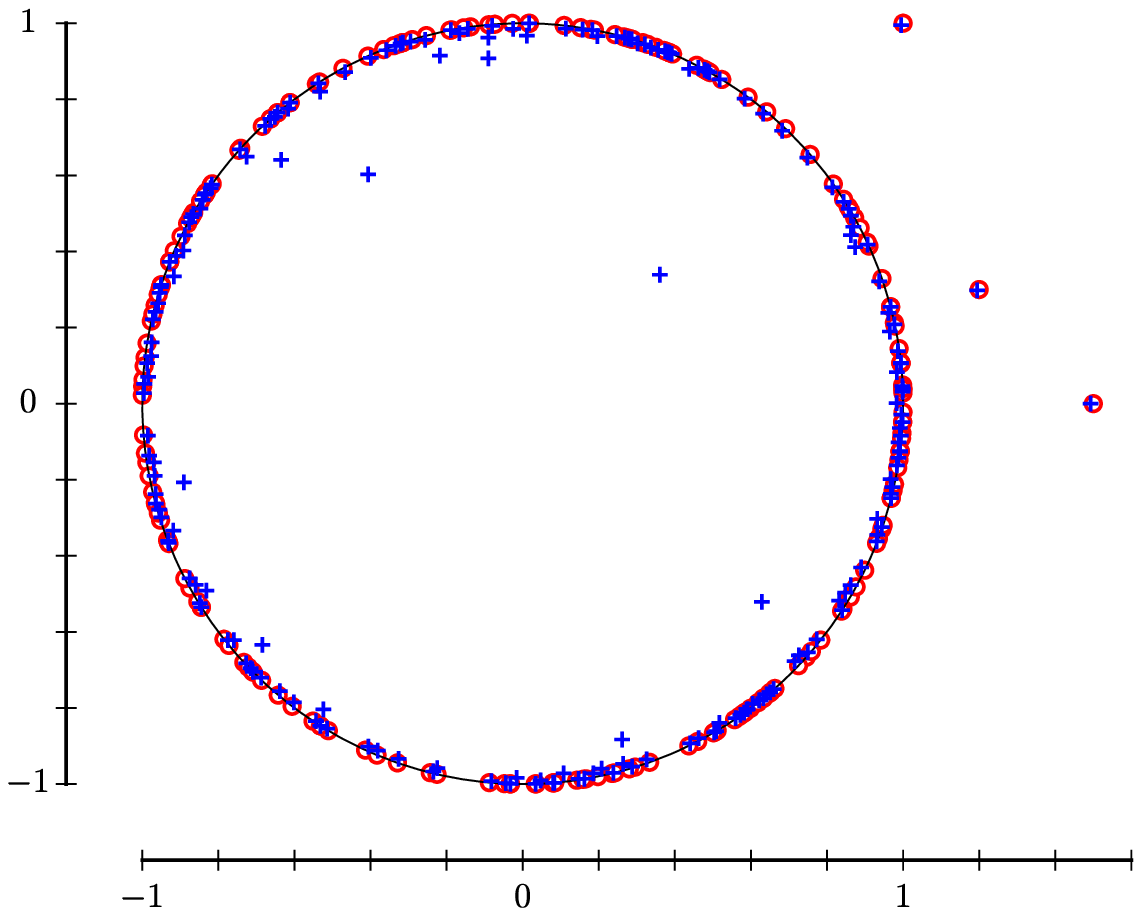}
\caption{The roots (red circles) and critical points (blue crosses) of a random, degree $203$ polynomial, where $200$ roots are chosen independently and uniformly on the unit circle (black curve), and three roots take the deterministic values $\xi_1 = 1 + i$, $\xi_2 = 1.5$, and $\xi_3 = 1.2+ 0.3i$.}
\label{fig:IIDCirc3DetOut}
\end{figure}

The goal of this note is to describe the pairing between the root $\xi$ and the nearby critical point.  More generally, we consider the case when several deterministic zeros are appended to the random polynomial $p_n$ and when $\mu$ is an arbitrary measure in the complex plane with compact support (not just the uniform distribution on the unit circle). See, for example, Figures \ref{fig:IIDCirc3DetOut} and \ref{fig:distinguished:IIDBlob}.

Let us mention that this pairing phenomenon between roots and critical points has been observed previously for random polynomials.  Hanin \cite{H2} proves a similar pairing result when a number of deterministic roots are appended to a random polynomial whose roots are chosen independently from a probability measure $\mu$ supported on the Riemann sphere.  Hanin's proof is guided by an intuitive electrostatic interpretation of the zeros and critical points.  In contrast to many of our results, Hanin's proof works both when $\mu$ is supported on a compact subset and when $\mu$ is supported on the entire Riemann sphere.  Unlike the results in \cite{H2} however, our results do not require the measure $\mu$ to have bounded density or require the deterministic roots to satisfy a separation condition.  In addition, our methods are significantly different than those used in \cite{H2} and allow us to describe the exact number of critical points lying in a region outside the support of $\mu$.  In a separate paper \cite{H1}, Hanin considers the joint distribution of roots and critical points for a class of Gaussian random polynomials.  However, the polynomials considered in \cite{H1} are quite different than the model considered in this paper.  Finally, let us mention the work of Dennis and Hannay \cite{DH} from the physics literature, which gives an electrostatic explanation for the pairing of critical points and zeros of random polynomials and characteristic polynomials of random matrices.

\subsection{Limiting distribution of the critical points}
To begin, we first consider the analogue of Theorem \ref{thm:PMK} when $o(n)$ deterministic zeros are appended to the random polynomial $p_n$ in \eqref{eq:pnprod}. 

\begin{theorem}[Limiting distribution of the critical points] \label{thm:genkabluchko}
Let $\mu$ be an arbitrary probability measure on $\mathbb{C}$, and suppose $X_1, X_2, \ldots$ are iid random variables with distribution $\mu$.  For each $n \geq 1$, let $k_n$ be a deterministic non-negative integer no larger than $n$ such that $k_n = o(n)$.  In addition, let $\xi_1^{(n)}, \ldots, \xi_{k_n}^{(n)}$ be a deterministic triangular array of complex values, and let
\[
p_n(z) := \prod_{j=1}^{n-{k_n}} (z - X_j) \prod_{l=1}^{k_n} (z - \xi_l^{(n)}).
\]
Then $\mu_{p_n}'$ converges weakly to $\mu$ in probability as $n \to \infty$.  
\end{theorem}

Theorem \ref{thm:genkabluchko} is a generalization of Theorem \ref{thm:PMK}.  Indeed, Theorem \ref{thm:PMK} can be recovered from Theorem \ref{thm:genkabluchko} by taking $k_n = 0$.  Unsurprisingly, we prove Theorem \ref{thm:genkabluchko} in Appendix \ref{sec:genkabluchko} by slightly generalizing the methods developed by Kabluchko in \cite{K}.  

Let us discuss the intuition behind Theorem \ref{thm:genkabluchko}.  To do so, we must begin with Theorem \ref{thm:PMK}.  Roughly speaking, Theorem \ref{thm:PMK} describes the phenomenon that if $p_n$ is a degree $n$ random polynomial, then 
\begin{equation} \label{eq:pnintuition}
	\mu_{p_n} - \mu'_{p_n} \longrightarrow 0 
\end{equation}
in probability as $n \to \infty$.  In other words, the limiting behavior of the critical points is the same as the limiting behavior of the roots.  While Theorem \ref{thm:PMK} only applies to random polynomials with iid roots, the same phenomenon has been observed for other ensembles of random polynomials \cite{O, OW}, and numerical simulations show that it should be true for many other models.  Stated another way, the behavior in \eqref{eq:pnintuition} appears to be universal among random polynomials.  Let us now consider the polynomial $p_n$ from Theorem \ref{thm:genkabluchko}.  It follows from the law of large numbers that $\mu_{p_n} \to \mu$ weakly almost surely as $n \to \infty$ since $k_n=o(n)$.  Therefore, if the convergence in \eqref{eq:pnintuition} applies to the polynomial $p_n$, the triangle inequality would immediately imply that $\mu_{p_n}'$ also converges weakly to $\mu$ in probability.  This heuristic is the basis for our proof of Theorem \ref{thm:genkabluchko}.

The above heuristic also hints that the condition $k_n = o(n)$ in Theorem \ref{thm:genkabluchko} is sharp.  Indeed, if $\lceil \eps n \rceil$ deterministic roots were to be appended, the limiting distribution is, in general, not $\mu$ as shown by the following example.
\begin{example}
Let $0 < \eps < 1$ and $k_n := \lceil\eps n \rceil$.  Define 
$$ p_n(z) := \prod_{j=1}^{n-k_n} (z - X_j), $$
where $X_1, X_2, \ldots$ are iid random variables uniformly distributed on the unit circle centered at the origin in the complex plane.  Then, by Theorem \ref{thm:PMK}, $\mu'_{p_n}$ converges weakly to the uniform measure on the unit circle in probability as $n \to \infty$.  However, the polynomial
$$ q_n(z) := z^{k_n} p_n(z) $$
has at least $k_n - 1$ critical points at the origin.  In particular, $\mu'_{q_n}(\{0\}) \geq \eps/2$ for $n$ sufficiently large.  Among other things, this implies that $\mu'_{q_n}$ does not converge weakly to the uniform probability measure on the unit circle as $n \to \infty$.  
\end{example}

While Theorem \ref{thm:genkabluchko} shows that the global behavior of the critical points is unchanged by the addition of $o(n)$ deterministic roots, the addition of one or more deterministic roots can create a number of outlying critical points as illustrated in Figures \ref{fig:IIDCircDetOut} and \ref{fig:IIDCirc3DetOut}.  One way of viewing this phenomenon is to view the deterministic roots as a small perturbation of the original polynomial.  This small perturbation is not enough to change the global distribution of the critical points; it may, however, as observed in the figures above, create a small number of outlying critical points.  Our main results below describe these outliers.  

\subsection{No outlying critical points for the unperturbed model}
Before we consider the perturbed model, we first consider the case when there are no deterministic roots.  In this initial case, we want to determine exactly where the critical points of the random polynomial $p_n$, defined in  \eqref{eq:pnprod}, are located.  This way, when we do append the small perturbation of deterministic roots, we will be able to tell exactly what effect the perturbation has had.  

Let $\mu$ be a probability measure on $\mathbb{C}$, and suppose $X_1, \ldots, X_n$ are iid random variables with distribution $\mu$.  In view of the Gauss--Lucas theorem (Theorem \ref{thm:gauss}), the roots of $p_n(z) := \prod_{j=1}^n (z - X_j)$, must lie in $\conv(\supp(\mu))$, the convex hull of the support of $\mu$. However, as we discussed above in the case when $\mu$ is supported on the unit circle (shown in Figure \ref{fig:IIDCirc}), nearly all of the critical points appear near the support of $\mu$, which is only a small subset of the convex hull.  Thus, our goal is to determine the exact subset of $\conv(\supp(\mu))$ where the critical points will lie, with high probability.  We do so in the theorem below.  To define this set where the critical points are located, we will first need to introduce the Cauchy--Stieltjes transform.  

Let $\mu$ be a probability measure on $\mathbb{C}$, and let $m_{\mu}$ be the Cauchy--Stieltjes transform of $\mu$ defined by
$$ m_{\mu}(z) := \int_{\mathbb{C}} \frac{ d \mu(x) }{z - x}, \quad z \not\in \supp(\mu). $$
Also, define
$$ \M_{\mu} := \left\{ z \in \mathbb{C} \setminus \supp(\mu) : m_{\mu}(z) = 0 \right\} $$
to be the set of zeros of $m_{\mu}$.  If $\mu$ has compact support, it turns out that $M_{\mu} \subset \conv(\supp(\mu))$; see Proposition \ref{prop:Mmu} for details.  For $\eps > 0$, we also define the set 
$$ \N_{\mu}(\eps) := \left\{ z \in \mathbb{C} : \dist(z, \supp(\mu) \cup M_{\mu}) < \eps \right\} $$
to be the $\eps$-neighborhood of $\supp(\mu) \cup M_{\mu}$.  Here, $\dist(z, D) := \inf_{w \in D} |z - w|$ is the distance from $z \in \mathbb{C}$ to a set $D \subset \mathbb{C}$.  

The following theorem shows that all critical points of $p_n$ must lie inside $\N_{\mu}(\eps)$  with high probability.  

\begin{theorem}[No outliers in the unperturbed model] \label{thm:nooutliers}
Let $\mu$ be a probability measure on $\mathbb{C}$ with compact support, and suppose $X_1, \ldots, X_n$ are iid random variables with distribution $\mu$.  Then, for every $\eps > 0$, there exists $C, c > 0$ (depending only on $\mu$ and $\eps$) such that, with probability at least $1 - C e^{-cn}$, the polynomial $p_n(z) := \prod_{j=1}^n(z - X_j)$ has no critical points outside $N_{\mu}(\eps)$.  
\end{theorem}

\begin{remark}
By the Gauss--Lucas theorem (Theorem \ref{thm:gauss}), the critical points of $p_n$ must lie inside $\conv(\supp(\mu))$.  Thus, Theorem \ref{thm:nooutliers} actually reveals that, with high probability, $p_n$ has no critical points outside $N_{\mu}(\eps) \cap \conv(\supp(\mu))$.  
\end{remark}

We now justify our choice of the set $\N_{\mu}(\eps)$ as the correct location of the critical points.  First, in the case that $\mu$ is degenerate, $p_n(z) = (z - a)^n$ for some $a \in \mathbb{C}$, which has critical point $z = a$ with multiplicity $n-1$.  This example shows that clearly the critical points of $p_n$ may lie in $\supp(\mu)$.  The next example shows that the critical points can also be in a neighborhood of the zero set $M_{\mu}$.  

\begin{example}
Let $\mu := p \delta_a + (1-p) \delta_b$ for some $a, b \in \mathbb{C}$ with $a \neq b$ and $p \in (0,1)$, and assume $X_1, X_2, \ldots$ are iid random variables with distribution $\mu$.  Then 
$$ p_n(z) := \prod_{j=1}^n (z - X_j) = (z - a)^\alpha (z-b)^\beta $$
for some non-negative integers $\alpha, \beta$ with $\alpha + \beta = n$.  Almost surely, for $n$ sufficiently large, $\alpha, \beta \geq 1$, and, in this case, 
$$ p_n'(z) = (z-a)^{\alpha-1} (z-b)^{\beta -1} \left(nz - \alpha b - \beta a \right).  $$
Thus, by the law of large numbers, $p_n$ has a critical point at
$$ z = \frac{\alpha b}{n} + \frac{\beta a}{n} = pb + (1-p) a + o(1) $$
almost surely.  On the other hand, 
$$ m_{\mu}(z) = \frac{p}{z-a} + \frac{1-p}{z-b} $$
has exactly one zero located at $z = pb + (1-p)a$.  
\end{example}

By the Borel--Cantelli lemma, Theorem \ref{thm:nooutliers} immediately implies the following corollary.  

\begin{corollary} \label{cor:nooutliers}
Let $\mu$ be a probability measure on $\mathbb{C}$ with compact support, and suppose $X_1, X_2, \ldots$ are iid random variables with distribution $\mu$.  Fix $\eps > 0$.  Then, almost surely, for $n$ sufficiently large, the polynomial $p_n(z) := \prod_{j=1}^n(z - X_j)$ has no critical points outside $N_{\mu}(\eps)$.  
\end{corollary}

We conclude this subsection with two examples of Theorem \ref{thm:nooutliers} and Corollary \ref{cor:nooutliers}.  

\begin{example}
Let $\mu$ be the uniform distribution on the unit circle centered at the origin.  A simple computation shows that
\[ m_\mu(z) = \left\{
     \begin{array}{rl}
       0, & \text{if } |z| < 1, \\
       \frac{1}{z}, & \text{if } |z| > 1,
     \end{array}
   \right. \]
and hence $M_{\mu} = \{z \in \mathbb{C} : |z| < 1 \}$.  Since $\conv(\supp(\mu)) = \{z \in \mathbb{C} : |z| \leq 1\}$, Theorem \ref{thm:nooutliers} does not rule out the possibility of critical points in the disk $D_{1 - \eps} := \{ z \in \mathbb{C} : |z| < 1 - \eps\}$.  This is not a limitation of Theorem \ref{thm:nooutliers} and is consistent with the results in \cite{PR}, which imply that, with positive probability, $D_{1 - \eps}$ contains at least one critical point.  More precisely, let $p_n(z) := \prod_{j=1}^n (z - X_j)$, where $X_1, X_2, \ldots$ are iid random variables with distribution $\mu$.  Then for any $0 < \eps < 1$, there exists $\eta > 0$ (independent of $n$) such that $p_n$ has a critical point in the disk $D_{1-\eps}$ with probability at least $\eta$ for all sufficiently large $n$.  This follows from the determinantal structure described in \cite[Theorem 3]{PR}.  A numerical simulation of this example is shown in Figure \ref{fig:IIDCirc}.  
\end{example}
 
\begin{example}
Let $\mu$ be the uniform distribution on the union of disjoint circles $C_1 \cup C_2$, where $C_1$ is the unit circle centered at $5/2$ and $C_2$ is the unit circle centered at $-5/2$.  Then
\[ m_\mu(z) = \left\{
     \begin{array}{ll}
       \frac{4z}{ 4z^2 - 25} , & \text{if } |z - 5/2| > 1 \text{ and } |z + 5/2| > 1, \\
       \frac{1}{2z + 5}, & \text{if } |z - 5/2| < 1, \\
       \frac{1}{2z - 5}, & \text{if } |z + 5/2| < 1,
     \end{array}
   \right. \]
and $M_\mu = \{0\}$.  Let $\eps > 0$, and take $p_n(z) := \prod_{j=1}^n (z - X_j)$, where $X_1, X_2, \ldots$ are iid random variables with distribution $\mu$.  Then Corollary \ref{cor:nooutliers} guarantees that almost surely, for $n$ sufficiently large, all critical points of $p_n$ lie in the set 
\[ A_1 \cup A_2 \cup \{z \in \mathbb{C} : |z| < \eps \}, \]
where $A_1$ and $A_2$ are the annuli
\[ A_1 := \{ z \in \mathbb{C} : 1 - \eps < |z - 5/2| < 1+ \eps \}, \quad A_2 := \{ z \in \mathbb{C} : 1 - \eps < |z + 5/2| < 1+ \eps \}. \]
A numerical simulation of this example is shown in Figure \ref{fig:TwoCirc}.  In particular, the simulation depicts a single critical point near the origin, showing that critical points may lie in a neighborhood of the zero set $M_\mu$.  In fact, it follows from the law of large numbers and Walsh's two circle theorem (see, for example, \cite[Theorem 4.1.1]{RS}) that, for any $0 < \eps < 1/4$, almost surely, for $n$ sufficiently large, there is exactly one critical point of $p_n$ in the disk $\{z \in \mathbb{C} : |z| < 1 + \eps\}$.  Combined with Corollary \ref{cor:nooutliers}, we conclude that almost surely this critical point must converge to the origin as $n$ tends to infinity.  
\end{example}

\begin{figure}
\includegraphics[width =.95\columnwidth]{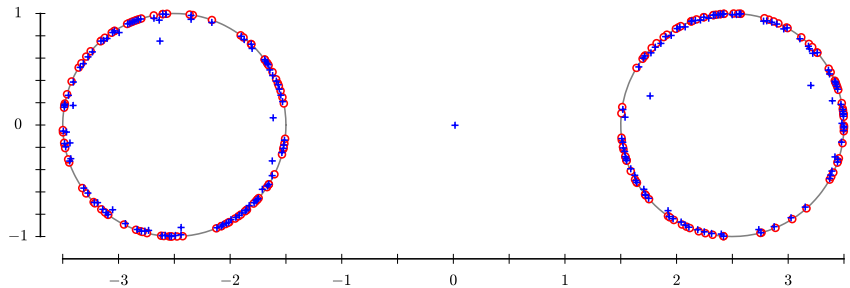}
\caption{The roots (red circles) and critical points (blue crosses) of a random, degree $200$ polynomial, where all $200$ roots are chosen independently and uniformly on the union of the two unit circles (black curves) centered at $-5/2$ and $5/2$, respectively.}
\label{fig:TwoCirc}
\end{figure}

\subsection{Locations of the outlying critical points in the perturbed model}
We now consider the outlying critical points depicted in Figures \ref{fig:IIDCircDetOut} and \ref{fig:IIDCirc3DetOut}.  To do  so, we will need the following notation.  For a polynomial $p$ of degree $n$, we let $w_1(p), \ldots, w_{n-1}(p)$ be the critical points of $p$ counted with multiplicity.  

\begin{theorem}[Locations of the outlying critical points] \label{thm:simple}
Let $\mu$ be a probability measure on $\mathbb{C}$ with compact support, and suppose $X_1, X_2, \ldots$ are iid random variables with distribution $\mu$.  Let $k \geq 1$, and assume $\xi_1, \ldots, \xi_k$ are deterministic complex numbers (which do not depend on $n$); in addition, suppose there are $s$ values $\xi_1, \ldots, \xi_s$ not in $\supp(\mu) \cup M_{\mu}$.  Then, there exists $\eps_0 > 0$ such that the following holds for any fixed $0 < \eps < \eps_0$.  Almost surely, for $n$ sufficiently large, there are exactly $s$ critical points (counted with multiplicity) of the polynomial 
$$ p_n(z) = \prod_{j=1}^{n-k} (z -X_j) \prod_{l=1}^k (z - \xi_l) $$
outside $N_{\mu}(\eps)$, and after labeling these critical points correctly, 
$$ w_{l}(p_n) = \xi_l + o(1) $$
for each $1 \leq l \leq s$.  
\end{theorem}

Theorem \ref{thm:simple} describes exactly the phenomenon we observe in Figures \ref{fig:IIDCircDetOut} and \ref{fig:IIDCirc3DetOut}.  In particular, this theorem shows that each deterministic root outside $\supp(\mu) \cup M_{\mu}$ creates one outlying critical point, which is asymptotically close to the deterministic root.  

For comparison, we provide the following example which shows that the conclusion of Theorem \ref{thm:simple} fails for deterministic polynomials.

\begin{example}
Let $p_n(z) := z^{n-1} - 1$ and $q_n(z) := p_n(z) (z - 1/2)$.  Then the roots of $q_n$ are $(n-1)$-th roots of unity with an outlier at $z = 1/2$.  However, we will show that $q_n$ has no critical points near $z = 1/2$.  Indeed, 
\[ q_n'(z) = n z^{n-1} - \frac{n-1}{2} z^{n-2} - 1, \]
and so the critical points are the solutions of
\[ 
\frac{1}{n} q_n'(z) = z^{n-1} - \frac{1}{2} \frac{n-1}{n} z^{n-2} - \frac{1}{n} = 0. \]
For $|z| \leq 3/4$, we have
\[
\left| z^{n-1} - \frac{1}{2} \frac{n-1}{n} z^{n-2} \right| \leq |z|^{n-1} + |z|^{n-2} \leq \frac{7}{4} \left( \frac{3}{4} \right)^{n-2} < \frac{1}{n}
\]
for $n$ sufficiently large.  This implies that $q_n'(z) \neq 0$ for every $z \in \mathbb{C}$ with $|z| \leq 3/4$.  Hence, for $n$ sufficiently large, there are no critical points of $q_n$ in the disk $\{z \in \mathbb{C} : |z| \leq 3/4\}$.  More generally, this argument shows that for a fixed $\eta \in (0,1)$, there are no critical points of $q_n$ in the disk $\{z \in \mathbb{C} : |z| \leq 1-\eta\}$ for sufficiently large $n$.
\end{example}

We next state two generalizations of Theorem \ref{thm:simple}.  Both results deal with the case when the deterministic points $\xi_1, \ldots, \xi_k$ (as well as the integer $k$) are allowed to depend on $n$.  Because the points can now depend on $n$, some additional technical assumptions are required.  These technical assumptions are trivially satisfied when $\xi_1, \ldots, \xi_k$ do not depend on $n$.  As such, Theorem \ref{thm:simple} is actually a corollary of the following more general result.  

\begin{theorem}[Locations of the outlying critical points: dependence on $n$] \label{thm:main}
Let $\mu$ be a probability measure on $\mathbb{C}$ with compact support, and suppose $X_1, X_2, \ldots$ are iid random variables with distribution $\mu$.  For each $n \geq 1$, let $\xi_1^{(n)}, \ldots, \xi_{k_n}^{(n)}$ be a triangular array of deterministic complex numbers with $k_n = O(1)$, and assume
\begin{equation} \label{eq:detmaxbnd}
	\max\{ |\xi_1^{(n)}|, \ldots, |\xi_{k_n}^{(n)}| \} = O(1).
\end{equation}
Fix $\eps > 0$, and suppose that for all sufficiently large $n$, there are no values of $\xi_1^{(n)}, \ldots, \xi_{k_n}^{(n)}$ in $\N_{\mu}(3 \eps) \setminus \N_{\mu}(\eps)$ and there are $s$ values $\xi_1^{(n)}, \ldots, \xi_{s}^{(n)}$ outside $\N_{\mu}(3 \eps)$.  Then, almost surely, for $n$ sufficiently large, there are exactly $s$ critical points (counted with multiplicity) of the polynomial 
$$ p_n(z) := \prod_{j=1}^{n-{k_n}} (z - X_j) \prod_{l=1}^{k_n} (z - \xi_l^{(n)}) $$ 
outside $\N_{\mu}(2 \eps)$, and after labeling these critical points correctly, 
$$ w_{l}(p_n) = \xi_l^{(n)} + o(1) $$
for each $1 \leq l \leq s$.  
\end{theorem}

The $O(1)$-magnitude assumption in \eqref{eq:detmaxbnd} is required for our proof.  However, we conjecture that this condition is not needed.  In fact, in the case when $s=1$, we can remove this assumption, and we obtain the following stronger result.  

\begin{theorem}[Locations of the outlying critical points: $s = 1$ case] \label{thm:distinguished}
Let $\mu$ be a probability measure on $\mathbb{C}$ with compact support, and suppose $X_1, X_2, \ldots$ are iid random variables with distribution $\mu$.  For each $n \geq 1$, let $\xi_1^{(n)}, \ldots, \xi_{k_n}^{(n)}$ be a triangular array of deterministic complex numbers with $k_n = O(1)$.  Fix $\eps > 0$, and suppose that for all sufficiently large $n$, there are no values of $\xi_1^{(n)}, \ldots, \xi_{k_n}^{(n)}$ in $\N_{\mu}(3 \eps) \setminus \N_{\mu}(\eps)$ and there is one value $\xi_1^{(n)}$ outside $\N_{\mu}(3 \eps)$.  Then, almost surely, for $n$ sufficiently large, there is exactly one critical point of the polynomial 
$$ p_n(z) := \prod_{j=1}^{n-{k_n}} (z - X_j) \prod_{l=1}^{k_n} (z - \xi_l^{(n)}) $$ 
outside $\N_{\mu}(2 \eps)$, and after labeling the critical points correctly, 
\begin{equation} \label{eq:location}
	w_{1}(p_n) = \xi_1^{(n)} \left( 1 + O \left( 1/n \right) \right) + O( 1/ n). 
\end{equation}
\end{theorem}
\begin{remark}
If $\xi_1^{(n)} = O(1)$, then \eqref{eq:location} implies that, almost surely, 
\[ w_1(p_n) = \xi_1^{(n)} + O(1/n). \]
More generally, if $\xi_1^{(n)} = o(n)$, Theorem \ref{thm:distinguished} yields that, almost surely, 
$$ w_1(p_n) = \xi_{1}^{(n)} + o(1). $$
In other words, the location of the outlying critical point $w_1(p_n)$ is asymptotically close to the outlying root $\xi_1^{(n)}$.  
\end{remark}

\begin{remark}
\label{rem:Hanin}
In the case where $\xi_1^{(n)}$ lies at least a fixed distance away from the convex hull of the support of $\mu$, the conclusion in \eqref{eq:location} is a deterministic result (regardless of the asymptotic behavior of $k_n$). This can be deduced from Walsh's two-circle theorem (see \cite[Theorem 4.1.1]{RS}).  
\end{remark}

We present a numerical simulation of Theorem \ref{thm:distinguished} in Figure \ref{fig:distinguished:IIDBlob}.

\begin{figure}
\includegraphics[width =.8\columnwidth]{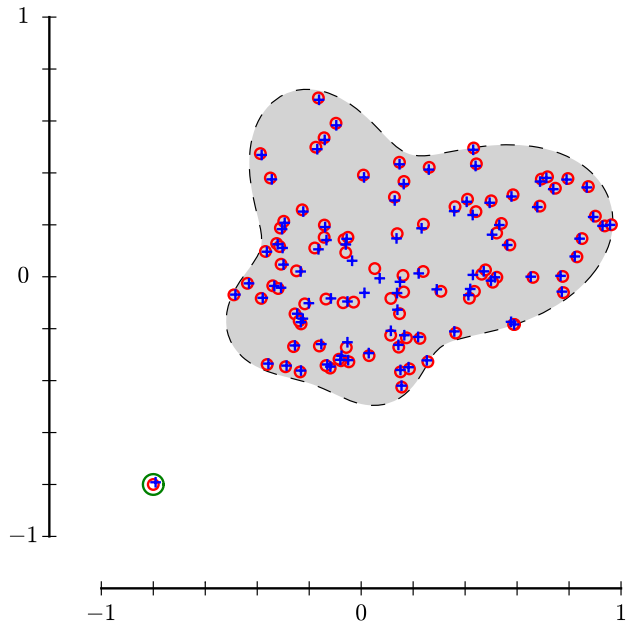}
\caption{The roots (red circles) and critical points (blue crosses) of a random, degree $n = 100$ polynomial, where $99$ roots are chosen independently and uniformly on the outlined region, and one root takes the deterministic value $\xi = -0.8 - 0.8i$. The small green circle centered at $\xi$ that contains the critical point nearby has radius $4/n$.}
\label{fig:distinguished:IIDBlob}
\end{figure}

\subsection{Outline}

The rest of the paper is devoted to the proof of our main results.  In Section \ref{sec:tools}, we develop several tools we will need for the proofs.  The proof of Theorem \ref{thm:nooutliers} is presented in Section \ref{sec:proof:nooutliers}.  We prove Theorems \ref{thm:simple}, \ref{thm:main}, and \ref{thm:distinguished} in Section \ref{sec:proofs}.  Finally, the proof of Theorem \ref{thm:genkabluchko} is given in Appendix \ref{sec:genkabluchko}.

\section{Tools and notation} \label{sec:tools}

We present here some tools we will need to prove our main results.

\subsection{Tools from probability theory}

We will need the following complex-valued version of Hoeffding's inequality.  

\begin{lemma}[Hoeffding's inequality for complex-valued random variables] \label{lemma:Hoeffding}
Let $Y_1, \ldots, Y_n$ be iid complex-valued random variables which satisfy $|Y_j| \leq K$ almost surely for some $K > 0$.  Then there exist absolute constants $C, c>0$ such that
$$ \Prob \left( \left| \frac{1}{n} \sum_{j=1}^n Y_j - \frac{1}{n} \E \left[ \sum_{j=1}^n Y_j \right] \right| \geq t \right) \leq C \exp \left( - c n t^2 / K^2 \right) $$
for every $t > 0$.  
\end{lemma}
\begin{proof}
Let 
$$ S_n := \frac{1}{n} \sum_{j=1}^n Y_j - \frac{1}{n} \E \left[ \sum_{j=1}^n Y_j \right]. $$  
If $|S_n| \geq t$, then $|\Re(S_n)| \geq t / \sqrt{2}$ or $|\Im(S_n)| \geq t / \sqrt{2}$.  So, we have
$$ \Prob( |S_n| \geq t) \leq \Prob ( |\Re(S_n)| \geq t / \sqrt{2} ) + \Prob(|\Im(S_n)| \geq t / \sqrt{2} ). $$
The claim now follows from the classic (real-valued) version of Hoeffding's inequality (see \cite{H}) since $|\Re(Y_j)| \leq K$ and $|\Im(Y_j)| \leq K$.  
\end{proof}

\subsection{Nets}

We introduce $\varepsilon$-nets as a convenient way to discretize a compact set.  

\begin{definition}
Let $X$ be a subset of $\mathbb{C}$, and $\eps > 0$.  A subset $\mathcal{N}$ of $X$ is called an $\eps$-net of $X$ if every point $x \in X$ can be approximated within $\eps$ by some point $y \in \mathcal{N}$, i.e. so that $|x - y| \leq \eps$.  
\end{definition}
For a finite set $\mathcal{N}$, we let $|\mathcal{N}|$ denote the cardinality of $\mathcal{N}$.  We will need the following estimate for the size of an $\eps$-net.  
\begin{lemma} \label{lemma:net}
Let $D$ be a compact subset of $\{z \in \mathbb{C} : |z| \leq M \}$ for some $M >0$.  Then, for every $\eps > 0$, there is an $\eps$-net $\mathcal{N}$ of $D$ such that 
$$ |\mathcal{N} | \leq \left( 1 + \frac{4 M}{\eps} \right)^2. $$
\end{lemma}
\begin{proof}
Let $\mathcal{N}'$ be a maximal $\eps/2$-separated subset of $S := \{z \in \mathbb{C} : |z| \leq M\}$.  In other words, $\mathcal{N}'$ is such that $|x-y| \geq \eps/2$ for all $x,y \in \mathcal{N}'$ with $x \neq y$, and no subset of $S$ containing $\mathcal{N}'$ has this property.  Such a set can always be constructed by starting with an arbitrary point in $S$ and at each step selecting a point that is at least $\eps/2$ distance away from those already selected.  Since $S$ is compact, this procedure will terminate after a finite number of steps.

The maximality property implies that $\mathcal{N}'$ is an $\eps/2$-net of $S$.  Indeed, otherwise there would exist $z \in S$ that is at least $\eps/2$-far from all points in $\mathcal{N}'$.  So $\mathcal{N}' \cup \{z\}$ would still be an $\eps/2$-separated set, contradicting the maximality property above.  

Moreover, the separation property implies that the balls of radii $\eps/4$ centered at the points in $\mathcal{N}'$ are disjoint.  In addition, all such balls lie in the ball of radius $M + \eps/4$ centered at the origin.  Comparing areas gives
$$ |\mathcal{N}'|  \left( \frac{\eps}{4} \right)^2 \leq \left( M + \frac{\eps}{4} \right)^2, $$
and hence 
$$ |\mathcal{N}'| \leq \left( 1 + \frac{4M}{\eps} \right)^2. $$
We now use $\mathcal{N}'$ to construct an $\eps$-net of $D$.  Indeed, we construct $\mathcal{N}$ iteratively using the following procedure.  Let $(x_n)_{n=1}^N$ be an enumeration of the points in $\mathcal{N}'$, and set $\mathcal{N}_0 := \emptyset$.  Given $\mathcal{N}_n$ for $0 \leq n \leq N-1$, we construct $\mathcal{N}_{n+1}$ as follows:
\begin{enumerate}
\item If the ball of radius $\eps/2$ centered at $x_{n+1}$ does not intersect $D$, then let $\mathcal{N}_{n+1} := \mathcal{N}_{n}$.  
\item If the ball of radius $\eps/2$ centered at $x_{n+1}$ does intersect $D$, let $y_{n+1}$ be an element of the intersection and set $\mathcal{N}_{n+1} := \mathcal{N}_n \cup \{y_{n+1} \}$.
\end{enumerate}
Now take $\mathcal{N} := \mathcal{N}_N$.  By the procedure above, it follows that $|\mathcal{N}| \leq |\mathcal{N}'|$.  It remains to show that $\mathcal{N}$ is an $\eps$-net of $D$.  Let $z \in D$.  Since $D \subseteq S$, there exists $x \in \mathcal{N}'$ such that $|x - z| \leq \eps/2$.  This means that the ball of radius $\eps/2$ centered at $x$ intersects $D$.  Thus, from the procedure above, there exists $y \in \mathcal{N}$ such that $|x - y| \leq \eps/2$.  Therefore, by the triangle inequality, $|z - y| \leq \eps$.  
\end{proof}

\subsection{Tools from linear algebra}

We will need the following companion matrix result, which describes a matrix whose eigenvalues are the critical points of a given polynomial.  This result appears to have originally been developed in \cite{KI} (see \cite[Lemma 5.7]{KI}).  However, the same result was later rediscovered and significantly generalized by Cheung and Ng \cite{CN,CN2}.

\begin{theorem}[Lemma 5.7 from \cite{KI}; Theorem 1.2 from \cite{CN2}] \label{thm:companion}
Let $p(z) := \prod_{j=1}^n (z - z_j)$ for some complex numbers $z_1, \ldots, z_n$, and let $D$ be the diagonal matrix $D := \diag(z_1, \ldots, z_n)$.  Then
$$ \frac{1}{n} z p'(z) = \det \left(zI - D \left( I - \frac{1}{n} J \right) \right), $$
where $I$ is the $n \times n$ identity matrix and $J$ is the $n \times n$ all-one matrix.  
\end{theorem}

Theorem \ref{thm:companion} allows us to translate the problem of studying critical points to a problem involving the eigenvalues of certain matrices.  For studying the eigenvalues of such matrices, we will need the following lemmata.

\begin{lemma}[Block determinant] \label{lemma:block}
Suppose $A,B,C$, and $D$ are matrices of dimension $n \times n$, $n \times m$, $m \times n$ and $m \times m$, respectively.  If $A$ is invertible, then
$$ \det \begin{pmatrix} A & B \\ C & D \end{pmatrix} = \det (A) \det (D - C A^{-1} B). $$
\end{lemma}
\begin{proof}
The conclusion follows immediately from the decomposition 
$$ \begin{pmatrix} A & B \\ C & D \end{pmatrix} = \begin{pmatrix} A & 0 \\ C & I_m \end{pmatrix} \begin{pmatrix} I_n & A^{-1} B \\ 0 & D - C A^{-1} B \end{pmatrix}, $$
where $I_n$ and $I_m$ are the identity matrices of dimension $n \times n$ and $m \times m$, respectively.  A similar proof is given in \cite[Section 0.8.5]{HJ}.  
\end{proof}


\begin{lemma}[Sherman--Morrison formula] \label{lemma:SM}
Suppose $A$ is an invertible matrix and $u,v$ are column vectors.  If ${1 + v^\mathrm{T} A^{-1} u \neq 0}$, then
$$ (A + uv^\mathrm{T})^{-1} = A^{-1} - \frac{A^{-1} u v^\mathrm{T} A^{-1}}{1 + v^\mathrm{T} A^{-1} u}. $$
\end{lemma}

Lemma \ref{lemma:SM} can be found in \cite{B}; see also \cite[Section 0.7.4]{HJ} for a more general version of this identity known as the Sherman--Morrison--Woodbury formula.  We will also require the following bound involving the difference of two determinants.  For a matrix $A$, we let $\|A\|$ denote the spectral norm of $A$, i.e.,\ $\|A\|$ is the largest singular value of $A$.  

\begin{lemma} \label{lemma:detdiff}
Let $A$ and $B$ be $k \times k$ matrices.  If $\|A\|, \|B\| = O(1)$, then
$$ \left| \det(A) - \det(B)\right| \ll_k \|A - B\|. $$
\end{lemma}
\begin{proof}
By the Leibniz formula for the determinant, it follows that
\begin{align}
	\left| \det(A) - \det(B) \right| &= \left| \sum_{\sigma} \sgn(\sigma) \left( \prod_{i=1}^k A_{\sigma(i), i} - \prod_{i=1}^k B_{\sigma(i), i} \right) \right| \nonumber \\
	&\leq \sum_{\sigma} \left| \prod_{i=1}^k A_{\sigma(i), i} - \prod_{i=1}^k B_{\sigma(i), i} \right|, \label{eq:detsum}
\end{align}
where the sums range over all permutations $\sigma$ of $\{1, \ldots, k\}$ and $\sgn(\sigma)$ is the sign of the permutation $\sigma$.  We now take advantage of the fact that the spectral norm of a matrix bounds the magnitude of each entry.  In particular, 
$$ \sup_{1 \leq i ,j \leq k} \left( |A_{ij}| + |B_{ij}| \right) \leq \|A\| + \|B\| = O(1) $$
and 
$$ \sup_{1 \leq i,j \leq k} |A_{ij} - B_{ij}| \leq \|A - B\|. $$
Thus, by multiple applications of the triangle inequality, we obtain
$$ \left| \prod_{i=1}^k A_{\sigma(i), i} - \prod_{i=1}^k B_{\sigma(i), i} \right| \ll_k \|A - B\| $$
uniformly in $\sigma$.  Combining this bound with \eqref{eq:detsum} completes the proof.  
\end{proof}

\subsection{Other tools}

We collect here some additional tools and facts we will need.  First, we note that if $\mu$ has compact support, then the convex hull of the support of $\mu$ is also a compact set; see \cite[Corollary 5.33]{IDA} for details.  

The following proposition shows that the zero set of the Cauchy--Stieltjes transform of $\mu$ must lie inside the convex hull of the support of $\mu$. It is a generalization of the Gauss--Lucas Theorem (Theorem \ref{thm:gauss}) in the sense that Proposition \ref{prop:Mmu} is precisely the Gauss--Lucas Theorem when $\mu$ is atomic.

\begin{proposition} \label{prop:Mmu}
Let $\mu$ be a probability measure on $\mathbb{C}$ with compact support.  If $m_{\mu}(z) = 0$ for some $z \not\in \supp(\mu)$, then $z \in \conv(\supp(\mu))$.  
\end{proposition}
\begin{proof}
Let $S := \conv(\supp(\mu))$, and define
\[
\overline{S} := \{ \overline{x} : x \in S \}.
\]
Suppose $z \not\in S$.  Then
\begin{align*}
	| m_{\mu}(z) | = | e^{i \theta} m_{\mu}(z) | \geq \left| \Im \left( e^{i \theta} m_{\mu}(z) \right) \right| = \left| \int_{\mathbb{C}} \frac{ \Im \left( e^{i\theta} (\overline{z} - \overline{x} ) \right) }{ |z- x|^2 } d\mu(x) \right| 
\end{align*}
for any $\theta \in \mathbb{R}$.  Since $\supp(\mu)$ is compact, it follows from \cite[Corollary 5.33]{IDA} that $\overline{S}$ is also compact.  Thus, by the hyperplane separation theorem, there exists a pair of parallel lines, separated by a gap $\eps > 0$, separating $\overline{S}$ and $\overline{z}$.  Let $\theta$ be the angle these lines make with the real axis (if they do not meet the real axis take $\theta = 0$).  Then $\Im \left( e^{i \theta} (\overline{z} - \overline{x}) \right)$ is of the same sign for all $x \in S$ and
$$ | \Im (e^{i \theta} (\overline{z} - \overline{x} )) | \geq \eps $$
for all $x \in S$.  Thus, we obtain
$$ | m_{\mu}(z) | \geq \eps \int_{\supp(\mu)} \frac{d \mu(x)}{|z - x|^2}. $$
As $\supp(\mu)$ is compact, there exists $M > 0$ such that $|z - x| \leq M$ for all $x \in S$.  Hence, we conclude that
$$ | m_{\mu}(z) | \geq \eps \frac{1}{M^2} > 0, $$
and the proof is complete.  
\end{proof}

We will also need the following observation concerning the translation of roots and critical points.  

\begin{proposition}[Translation of the critical points] \label{prop:translate}
Let $p$ be a monic polynomial of degree $n$, and suppose $w_1, \ldots, w_{n-1}$ are the critical points of $p$ counted with multiplicity.  Then, for any $a \in \mathbb{C}$, the critical points of $q(z) := p(z-a)$ are $w_{1} + a, \ldots, w_{n-1} + a$.  
\end{proposition}
\begin{proof}
Since $p$ is a monic polynomial of degree $n$, 
$$ p'(z) = n \prod_{j=1}^{n-1}(z - w_j). $$
Thus,
$$ q'(z) = p'(z-a) = n \prod_{j=1}^{n-1} (z - a - w_{j}), $$
and the claim follows.  
\end{proof}

\section{Proof of Theorem \ref{thm:nooutliers}} \label{sec:proof:nooutliers}

This section is devoted to the proof of Theorem \ref{thm:nooutliers}.  For $\eps > 0$, define 
$$ \Su_{\mu}(\eps) := \left\{ z \in \mathbb{C} : \dist(z, \supp(\mu)) < \eps \right\} $$
to be the $\eps$-neighborhood of the support of $\mu$.  
We begin with the following concentration inequality.  
\begin{lemma} \label{lemma:lln}
Let $\mu$ be a probability measure on $\mathbb{C}$ with compact support, and suppose $X_1, \ldots, X_n$ are iid random variables with distribution $\mu$.  Then, for every $M, \eps, t > 0$, 
$$ \Prob \left( \sup_{{ z \in \mathbb{C} : |z| \leq M, z \not\in \Su_{\mu}(\eps) }} \left| \frac{1}{n} \sum_{j=1}^n \frac{1}{z - X_j} - m_{\mu}(z) \right| \geq t \right) \leq C \left( 1 + \frac{40M}{\eps^2 t} \right) \exp \left(-cn t^2 \eps^2 \right) $$
for some absolute constants $C,c > 0$.      
\end{lemma}
\begin{proof}
Let $M, \eps, t > 0$, and define
$$ D:= \{z \in \mathbb{C} : |z| \leq M, z \not\in \Su_{\mu}(\eps) \}. $$
We assume $D$ is nonempty as the conclusion is trivial otherwise. Let $\mathcal{N}$ be an $\eps^2 t / 10$-net of $D$.  By Lemma \ref{lemma:net}, $\mathcal{N}$ can be chosen so that
\begin{equation} \label{eq:netbnd}
	|\mathcal{N}| \leq \left( 1 + \frac{40 M}{\eps^2 t } \right)^2. 
\end{equation}

We observe that $X_j \in \supp(\mu)$ almost surely for every $1 \leq j \leq n$.  Thus, almost surely, for $z \in D$, 
\begin{equation} \label{eq:bndonS}
	\left| \frac{1}{z - X_j} \right| \leq \frac{1}{\eps}. 
\end{equation}
Hence, for $z, w \in D$, 
$$ \left| \frac{1}{n} \sum_{j=1}^n \frac{1}{z - X_j} - \frac{1}{n} \sum_{j=1}^n \frac{1}{w - X_j} \right| \leq \frac{|z-w|}{\eps^2}. $$
In other words, the function $m_n(z) := \frac{1}{n} \sum_{j=1}^n \frac{1}{z - X_j}$ is almost surely Lipschitz continuous on $D$ with Lipschitz constant $\eps^{-2}$.  Similarly, for $z,w \in D$,
\begin{align*}
	|m_{\mu}(z) - m_{\mu}(w)| &= \left| \int_{\supp(\mu)} \left( \frac{1}{z - x} - \frac{1}{w-x} \right) d \mu(x) \right| \\
	&\leq \int_{\supp(\mu)} \frac{|z-w|}{|z-x| |w-x| } d\mu(x) \\
	&\leq \frac{|z-w|}{\eps^2}. 
\end{align*}

Suppose $\sup_{z \in D} \left| m_n(z) - m_{\mu}(z) \right| \geq t$.  As $m_n$ and $m_{\mu}$ are both continuous on the compact set $D$, there exists $z \in D$ such that $|m_n(z) - m(z)| \geq t$.  Since $\mathcal{N}$ is an $\eps^2 t / 10$-net of $D$, there exists $w \in \mathcal{N}$ such that $|z-w| \leq \frac{\eps^2 t}{10}$.  So, by the reverse triangle inequality and the fact that the $m_n$ and $m_{\mu}$ are Lipschitz continuous, we have 
\begin{align*}
	|m_n(w) - m_{\mu}(w)| &\geq |m_n(z) - m_{\mu}(z)| - |m_n(z) - m_n(w) - (m_{\mu}(z) - m_{\mu}(w))| \\
	&\geq t - 2 \frac{|z-w|}{\eps^2} \\
	&\geq \frac{4t}{5}. 
\end{align*}
Therefore, by the union bound, we conclude that
\begin{align} 
	\Prob \left( \sup_{z \in D} | m_n(z) - m_{\mu}(z)| \geq t \right) &\leq \Prob \left( \sup_{w \in \mathcal{N}} |m_n(w) - m_{\mu}(w)| \geq \frac{4t}{5} \right) \nonumber \\
	&\leq \sum_{w \in \mathcal{N}} \Prob \left( |m_n(w) - m_{\mu}(w)| \geq \frac{4t}{5} \right) \label{eq:probbndzinD}
\end{align}
As $\E m_n(z) = m_{\mu}(z)$ for $z \in D$, Hoeffding's inequality (Lemma \ref{lemma:Hoeffding}) and the bound in \eqref{eq:bndonS} imply that
\begin{equation} \label{eq:supprobwbnd}
	\sup_{w \in \mathcal{N} } \Prob \left( |m_n(w) - m_{\mu}(w)| \geq \frac{4t}{5} \right) \leq C \exp(-cn t^2 \eps^2) 
\end{equation}
for some absolute constants $C, c > 0$.  Thus, combining \eqref{eq:netbnd}, \eqref{eq:probbndzinD}, and \eqref{eq:supprobwbnd} yields 
$$ \Prob \left( \sup_{z \in D} | m_n(z) - m_{\mu}(z)| \geq t \right) \leq C \left( 1 + \frac{40M}{\eps^2 t} \right) \exp(-c n t^2 \eps^2), $$
as desired.  
\end{proof}

We now prove Theorem \ref{thm:nooutliers}.

\begin{proof}[Proof of Theorem \ref{thm:nooutliers}]
Let $\eps > 0$.  With probability one, $X_j \in \supp(\mu)$ for each $1 \leq j \leq n$.  Thus, the zeros of 
$$ m_n(z) := \frac{1}{n} \frac{p_n'(z) }{p_n(z)} = \frac{1}{n} \sum_{j=1}^n \frac{1}{z - X_j} $$
outside of $\N_{\mu}(\eps)$ are exactly the critical points of $p_n$ outside of $\N_{\mu}(\eps)$.  We will show that $m_n(z)$ has no zeros in $D := \conv(\supp(\mu)) \setminus \N_{\mu}(\eps)$.  The claim then follows immediately since, by the Gauss--Lucas theorem (Theorem \ref{thm:gauss}), all the critical points of $p_n$ lie in $\conv(\supp(\mu))$.  

Since $\mu$ has compact support, $\conv(\supp(\mu))$ is also a compact set (see \cite[Corollary 5.33]{IDA}), and hence $D$ is compact.  As $m_{\mu}$ is a continuous function on $D$, $|m_{\mu}|$ achieves its minimum on $D$, which, by definition of $\N_{\mu}(\eps)$ cannot be zero (since $\N_{\mu}(\eps)$ contains the zero set $\M_{\mu}$).  Thus, there exists $c' > 0$ such that 
$$ |m_{\mu}(z)| \geq c'  \quad \text{for all } z \in D. $$

Since $D$ is compact, there exists $M > 0$ (depending only on $\supp(\mu)$) such that $|z| \leq M$ for all $z \in D$.  Thus, by Lemma \ref{lemma:lln} (taking $t = c'/2$), we obtain
$$ \Prob \left( \sup_{z \in D} \left| m_n(z) - m_{\mu}(z) \right| \geq \frac{c'}{2} \right) \leq C \left( 1 + \frac{80M}{\eps^2 c'} \right) \exp(-cn c' \eps^2) $$ 
for some absolute constants $C, c >0$.  Hence, on the complementary event, we have
$$ |m_n(z)| \geq |m_{\mu}(z)| - |m_n(z) - m_{\mu}(z)| \geq \frac{c'}{2} $$
for all $z \in D$.  Since the constants $C \left( 1 + \frac{80M}{\eps^2 c'} \right)$ and $c c' \eps^2$ only depend on $\mu$ and $\eps$, the proof is complete.  
\end{proof}

\section{Proof of Theorems \ref{thm:simple}, \ref{thm:main}, and \ref{thm:distinguished}} \label{sec:proofs}

This section is devoted to the proof of Theorems \ref{thm:simple}, \ref{thm:main}, and \ref{thm:distinguished}.

\subsection{Proof of Theorem \ref{thm:simple}}

We now prove Theorem \ref{thm:simple} using Theorem \ref{thm:main}.  Indeed, let $\xi_1, \ldots, \xi_k$ satisfy the assumptions of Theorem \ref{thm:simple}.  Since $\xi_1, \ldots, \xi_k$ do not depend on $n$, there exists $\eps_0 > 0$ such that, for any $0 < \eps < \eps_0$, 
\begin{itemize}
\item $\xi_1, \ldots, \xi_s$ are outside $\N_{\mu}(3 \eps)$,
\item $\xi_{s+1}, \ldots, \xi_{k}$ are in $\N_{\mu}(\eps)$.
\end{itemize}
In addition, condition \eqref{eq:detmaxbnd} trivially holds because $\xi_1, \ldots, \xi_k$ do not depend on $n$.  Thus, Theorem \ref{thm:main} is applicable for any $0 < \eps < \eps_0$, and hence Theorem \ref{thm:simple} follows.

\subsection{Proof of Theorems \ref{thm:main} and \ref{thm:distinguished}}

We will prove Theorem \ref{thm:main} via the following result.  

\begin{theorem} \label{thm:reduction}
Let $\mu$ be a probability measure on $\mathbb{C}$ with compact support, and suppose $0 \in \supp(\mu)$.  Let $X_1, X_2, \ldots$ be iid random variables with distribution $\mu$.  For each $n \geq 1$, let $\xi_1^{(n)}, \ldots, \xi_{k_n}^{(n)}$ be a triangular array of deterministic complex numbers with $k_n = O(1)$, and assume $\max\{ |\xi_1^{(n)}|, \ldots, |\xi_{k_n}^{(n)}| \} = O(1)$.  Fix $\eps > 0$, and suppose that for all sufficiently large $n$, there are no values of $\xi_1^{(n)}, \ldots, \xi_{k_n}^{(n)}$ in $\N_{\mu}(3 \eps) \setminus \N_{\mu}(\eps)$ and there are $s$ values $\xi_1^{(n)}, \ldots, \xi_{s}^{(n)}$ outside $\N_{\mu}(3 \eps)$.  Then, almost surely, for $n$ sufficiently large, there are exactly $s$ critical points (counted with multiplicity) of the polynomial 
$$ p_n(z) := \prod_{j=1}^{n-{k_n}} (z - X_j) \prod_{l=1}^{k_n} (z - \xi_l^{(n)}) $$ 
outside $\N_{\mu}(2 \eps)$, and after labeling these critical points correctly, 
$$ w_{l}(p_n) = \xi_l^{(n)} + o(1) $$
for each $1 \leq l \leq s$.  
\end{theorem}

The only difference between this theorem and Theorem \ref{thm:main} is that Theorem \ref{thm:reduction} assumes $0 \in \supp(\mu)$.  Using Theorem \ref{thm:reduction}, we prove Theorem \ref{thm:main} by applying Proposition \ref{prop:translate}.  

\begin{proof}[Proof of Theorem \ref{thm:main}]
Let $\mu$ have compact support.  Since $\supp(\mu)$ is nonempty, choose $a \in \supp(\mu)$.  We now consider the polynomial 
$$ p_n(z + a) = \prod_{j=1}^{n-k_n} (z - (X_j - a) ) \prod_{l=1}^{k_n} ( z - (\xi_l^{(n)} - a)) = \prod_{j=1}^{n-k_n}(z - Y_j) \prod_{l=1}^{k_n} (z - (\xi_l^{(n)} - a)), $$
where $Y_j := X_j - a$.  Let $\nu$ be the distribution of $Y_1$.  Then $\nu$ has compact support and $0 \in \supp(\nu)$.  In addition, the sets $\M_{\nu}$ and $\supp(\nu)$ are translates by $-a$ of the sets $\M_{\mu}$ and $\supp(\mu)$, respectively.  Thus, by assumption, there are no values of $\xi_1^{(n)} - a, \ldots, \xi_{k_n}^{(n)} - a$ in $\N_{\nu}(3\eps) \setminus \N_{\nu}(\eps)$ and there are $s$ values $\xi_1^{(n)} - a, \ldots, \xi_s^{(n)} - a$ outside $\N_{\nu}( 3\eps)$.  Therefore, by Theorem \ref{thm:reduction} and Proposition \ref{prop:translate}, we conclude that almost surely, for $n$ sufficiently large, there are exactly $s$ critical points of $p_n$ outside $\N_{\mu}(2 \eps)$ and after labeling correctly, 
$$ w_{l}(p_n) - a = \xi_l^{(n)} - a  + o(1)  $$
for $1 \leq l \leq s$.  Adding $a$ to both sides completes the proof.  
\end{proof}

Similarly, Theorem \ref{thm:distinguished} can be proven using the following.  

\begin{theorem} \label{thm:reductiondist}
Let $\mu$ be a probability measure on $\mathbb{C}$ with compact support, and suppose $0 \in \supp(\mu)$.  Let $X_1, X_2, \ldots$ be iid random variables with distribution $\mu$.  For each $n \geq 1$, let $\xi_1^{(n)}, \ldots, \xi_{k_n}^{(n)}$ be a triangular array of deterministic complex numbers with $k_n = O(1)$.  Fix $\eps > 0$, and suppose that for all sufficiently large $n$, there are no values of $\xi_1^{(n)}, \ldots, \xi_{k_n}^{(n)}$ in $\N_{\mu}(3 \eps) \setminus \N_{\mu}(\eps)$ and there is one value $\xi_1^{(n)}$ outside $\N_{\mu}(3 \eps)$.  Then, almost surely, for $n$ sufficiently large, there is exactly one critical point of the polynomial 
$$ p_n(z) := \prod_{j=1}^{n-{k_n}} (z - X_j) \prod_{l=1}^{k_n} (z - \xi_l^{(n)}) $$ 
outside $\N_{\mu}(2 \eps)$, and after labeling the critical points correctly, 
$$ w_{1}(p_n) = \xi_1^{(n)} \left( 1 + O \left( 1/n \right) \right). $$
\end{theorem}

The proof of Theorem \ref{thm:distinguished} using Theorem \ref{thm:reductiondist} is nearly identical to the proof of Theorem \ref{thm:main} above; we omit the details.  It remains to prove Theorems \ref{thm:reduction} and \ref{thm:reductiondist}.

\subsection{Proof of Theorems \ref{thm:reduction} and \ref{thm:reductiondist}}

We prove Theorems \ref{thm:reduction} and \ref{thm:reductiondist} simultaneously.  Indeed, for the first part of the proof, we continue to use the notation of Theorem \ref{thm:reduction}.  However, the same argument applies to Theorem \ref{thm:reductiondist} by simply taking $s=1$.  The conclusion of the proof will require us to consider the conditions of both theorems separately.  In fact, the conclusion of the proof is the only place where we require condition \eqref{eq:detmaxbnd}.  For notational convenience, throughout the proof we allow the implicit constants and rates of convergence in our asymptotic notation (such as $O, o$) to depend on the parameter $\eps$ without notating this dependence.  

For $n$ sufficiently large, we decompose
$$ p_n(z) = \prod_{j=1}^{n-{k_n}} (z - X_j) \prod_{l=1}^s (z - \xi_{l}^{(n)}) \prod_{l={s+1}}^{k_n} (z - \xi_l^{(n)}), $$
where, by assumption, $\xi_{1}^{(n)}, \ldots, \xi_s^{(n)}$ are outside $\N_{\mu}(3 \eps)$ and $\xi_{s+1}^{(n)}, \ldots, \xi_{k_n}^{(n)}$ are in $\N_{\mu}(\eps)$.  In addition, $X_1, \ldots, X_{n-k_n}$ are in $\supp(\mu) \subset \N_{\mu}(\eps)$ with probability $1$.

Let $D$ be the diagonal matrix 
$$ D := \begin{pmatrix} D_\In & 0 \\ 0 & D_\Out \end{pmatrix}, $$
where
$$ D_\In := \diag(X_1, \ldots, X_{n- k_n}, \xi_{s+1}^{(n)}, \ldots, \xi_{k_n}^{(n)}) $$
and
$$ D_\Out := \diag(\xi_{1}^{(n)}, \ldots, \xi_{s}^{(n)}). $$
Here, the subscripts ``$\In$'' and ``$\Out$'' refer to the roots inside and outside $\N_{\mu}(\eps)$, respectively.  Of course, $D$, $D_\In$, and $D_\Out$ all depend on $n$, but we do not denote this dependence in our notation.  

By Theorem \ref{thm:companion}, it follows that 
\begin{align}
	\frac{1}{n} z p_n'(z) &= \det \left( zI - D + \frac{1}{n} DJ_n \right) \label{eq:firstdeteq} \\
	&= \det \left[ \begin{pmatrix} z I & 0 \\ 0 & z I \end{pmatrix} - \begin{pmatrix} D_\In & 0 \\ 0 & D_\Out \end{pmatrix} + \frac{1}{n} \begin{pmatrix} D_\In & 0 \\ 0 & D_\Out \end{pmatrix} J_n \right], \nonumber
\end{align}
where $I$ is the identity matrix and $J_n$ is the $n \times n$ all-one matrix.  We decompose, 
$$ J_n = \begin{pmatrix} J_{n-s} & J_{n-s, s} \\ J_{s, n-s} & J_{s} \end{pmatrix}, $$
where $J_{l,m}$ denotes the $l \times m$ all-one matrix.  Thus, we conclude that
\begin{equation} \label{eq:blockdet}
	\frac{1}{n} z p_n'(z) = \det \begin{pmatrix} zI - D_\In + \frac{1}{n} D_\In J_{n-s} & \frac{1}{n} D_\In J_{n-s,s} \\ \frac{1}{n} D_\Out J_{s,n-s} & zI - D_\Out + \frac{1}{n} D_\Out J_{s} \end{pmatrix}. 
\end{equation}
We will eventually apply Lemma \ref{lemma:block} to compute this determinant, but first we will need to consider the upper-left block
$$ zI - D_\In + \frac{1}{n} D_\In J_{n-s}. $$
Let $\Bj_n$ denote the all-one $n$-vector; we will often drop the subscript (and just write $\Bj$) when its size can be deduced from context.  We will make use of the following lemma.  

\begin{lemma} \label{lemma:technical}
Under the assumptions of Theorem \ref{thm:reduction} (alternatively, Theorem \ref{thm:reductiondist}), almost surely, for $n$ sufficiently large, the matrix 
\begin{equation} \label{eq:zIDmatrix}
	zI - D_\In + \frac{1}{n} D_\In J_{n-s} 
\end{equation}
is invertible for every $z \not\in \N_{\mu}(2 \eps)$ and the function 
\begin{equation} \label{eq:analytic}
	z \mapsto \frac{1}{n} \Bj^\mathrm{T} \left( zI - D_\In + \frac{1}{n} D_\In J_{n-s} \right)^{-1} D_{\In} \Bj 
\end{equation}
is analytic outside $\N_{\mu}(2\eps)$.  In addition, almost surely 
$$ \sup_{z \in \mathbb{C} \setminus \N_{\mu}(2 \eps)} \left| \frac{1}{n} \Bj^\mathrm{T} \left( zI - D_\In + \frac{1}{n} D_\In J_{n-s} \right)^{-1} D_{\In} \Bj \right| = O(1). $$
\end{lemma}
\begin{proof}
Recall that the entries of the diagonal matrix $D_\In$ are contained in $\N_{\mu}(\eps)$.  Thus, for $z \not\in \N_{\mu}(2 \eps)$, the matrix $zI - D_\In$ is invertible.  In addition, since $(zI - D_\In)^{-1}$ is a diagonal matrix, we obtain
\begin{align}
	\frac{1}{n} \Bj^\mathrm{T} ( zI - D_\In)^{-1} D_\In \Bj &= \frac{1}{n} \tr [(zI - D_\In )^{-1} D_\In]  \nonumber \\
	&= \frac{1}{n} \sum_{j=1}^{n-k_n} \frac{X_j}{z - X_j} + \frac{1}{n} \sum_{l=s+1}^{k_n} \frac{\xi_l^{(n)}}{ z - \xi_l^{(n)}}.  \label{eq:trpartial}
\end{align}
Among other things, this implies that the function $\frac{1}{n} \Bj^\mathrm{T} ( zI - D_\In)^{-1} D_\In \Bj$ is analytic outside $\N_{\mu}(2\eps)$; we will use this fact later to show that the function in \eqref{eq:analytic} is analytic on the same set.    
Since $\supp(\mu)$ is compact, it follows from Proposition \ref{prop:Mmu} that $\N_{\mu}(\eps)$ is bounded.  Let $\kappa > 0$ be such that $|z| \leq \kappa$ for all $z \in \N_{\mu}(\eps)$.  Let $M := 10 \kappa$.  Then for $|z| \geq M$, we have
$$ |X_j| \leq \kappa, \quad |z - X_j| \geq M - \kappa = 9 \kappa $$
for $1 \leq j \leq n-k_n$ and similarly 
$$ |\xi_l^{(n)}| \leq \kappa, \quad |z - \xi_l^{(n)}| \geq M - \kappa = 9 \kappa $$
for each $s+1 \leq l \leq k_n$.  Thus, 
\begin{equation} \label{eq:bigzbnd}
	\sup_{|z| \geq M} \left| \frac{1}{n} \Bj^\mathrm{T} ( zI - D_\In)^{-1} D_\In \Bj  \right| \leq \frac{\kappa}{9 \kappa} = \frac{1}{9}. 
\end{equation}
In particular, this bound implies that $1 +  \frac{1}{n} \Bj^\mathrm{T} ( zI - D_\In)^{-1} D_\In \Bj \neq 0$ for all $|z| \geq M$.  Thus, we can apply Lemma \ref{lemma:SM} to conclude that the matrix in \eqref{eq:zIDmatrix} is invertible for every $|z| \geq M$.  Indeed, since $\frac{1}{n} D_\In J_{n-s} = \frac{1}{n} D_\In \Bj \Bj^\mathrm{T}$ is at most rank one\footnote{Here, we have used the fact that $J_{n-s}$ is rank one, and so the product $D_\In J_{n-s}$ is either rank one or rank zero.  In fact, a simple computation reveals that the product is rank zero if and only if $D_\In$ is the zero matrix.}, it follows from Lemma \ref{lemma:SM} (taking $u = D_\In\Bj$ and $v = \Bj$) that
\begin{align}
	\frac{1}{n} \Bj^\mathrm{T} &\left( zI - D_\In + \frac{1}{n} D_\In J_{n-s} \right)^{-1} D_\In \Bj \nonumber \\
		&= \frac{1}{n} \Bj^\mathrm{T} (zI - D_\In)^{-1} D_\In \Bj - \frac{ \left( \frac{1}{n} \Bj^\mathrm{T} (zI - D_\In)^{-1} D_\In \Bj \frac{1}{n} \right)^2}{1 + \frac{1}{n} \Bj^\mathrm{T} ( zI - D_\In)^{-1} D_\In \Bj }. \label{eq:SMin}
\end{align}
Hence, by the bound in \eqref{eq:bigzbnd}, we have, with probability one, 
$$ \sup_{|z| \geq M} \left| \frac{1}{n} \Bj^\mathrm{T} \left( zI - D_\In + \frac{1}{n} D_\In J_{n-s} \right)^{-1} D_\In \Bj \right| \leq \frac{1}{9} + \frac{ \left(\frac{1}{9}\right)^2}{1 - \frac{1}{9}} = O(1). $$
In addition, the right-hand side of \eqref{eq:SMin} is analytic in the region $|z| \geq M$, which implies that the function on the left-hand side is also analytic in the same region.  

Let $\Omega$ be the compact set $\{z \in \mathbb{C} : |z| \leq M\} \setminus \N_{\mu}(2\eps)$.  It remains to show that, almost surely, for $n$ sufficiently large, the matrix in \eqref{eq:zIDmatrix} is invertible for every $z \in \Omega$, the function in \eqref{eq:analytic} is analytic in $\Omega$, and
$$ \sup_{z \in \Omega} \left| \frac{1}{n} \Bj^\mathrm{T} \left( zI - D_\In + \frac{1}{n} D_\In J_{n-s} \right)^{-1} D_\In \Bj \right| = O(1). $$
To establish these results we will again apply Lemma \ref{lemma:SM}.  However, in this case, we will need more precise estimates than those established above.  

Indeed, returning to \eqref{eq:trpartial}, we find that
\begin{align} \label{eq:identfullexpand}
	\frac{1}{n} \Bj^\mathrm{T} ( zI - D_\In)^{-1} D_\In \Bj &= - \frac{n-s}{n} + \frac{z}{n} \sum_{j=1}^{n-k_n} \frac{1}{z - X_j} + \frac{z}{n} \sum_{l=s+1}^{k_n} \frac{1}{z - \xi_{l}^{(n)}}.
\end{align}
Since $\xi_{s+1}^{(n)}, \ldots, \xi_{k_n}^{(n)}$ are contained in $\N_{\mu}(\eps)$, it follows from the triangle inequality that
\begin{equation} \label{eq:bnddetin}
	\sup_{z \in \Omega} \left | \frac{z}{n} \sum_{l=s+1}^{k_n} \frac{1}{z - \xi_{l}^{(n)}} \right| \leq \frac{k_n}{n} \frac{M}{\eps} = o(1). 
\end{equation}
In addition, by Lemma \ref{lemma:lln} and the Borel--Cantelli lemma, we have, almost surely
\begin{equation} \label{eq:bndrndin}
	\sup_{z \in \Omega} \left| \frac{z}{n} \sum_{j=1}^{n-k_n} \frac{1}{z - X_j} - z m_{\mu}(z) \right| = o(1). 
\end{equation}
As $\Omega$ is compact and $m_{\mu}$ cannot vanish on $\Omega$ (since $M_{\mu} \subset N_{\mu}(\eps)$), there exists $C, c > 0$ such that $c \leq |m_{\mu}(z)| \leq C$ for all $z \in \Omega$.  Specifically, by the assumption that $0 \in \supp(\mu)$, it follows that
\begin{equation} \label{eq:lowupbndzmu}
	\eps c \leq |z m_{\mu}(z)| \leq M C, \quad \text{ for all } z \in \Omega. 
\end{equation}

Therefore, by \eqref{eq:bnddetin}, \eqref{eq:bndrndin}, and \eqref{eq:lowupbndzmu}, we conclude from \eqref{eq:identfullexpand} that, almost surely, for $n$ sufficiently large, 
$$ \sup_{z \in \Omega} \left| \frac{1}{n} \Bj^\mathrm{T} ( zI - D_\In)^{-1} D_\In \Bj \right| \leq 2 + MC $$
and
$$ \inf_{z \in \Omega} \left| 1 + \frac{1}{n} \Bj^\mathrm{T} ( zI - D_\In)^{-1} D_\In \Bj \right| \geq \frac{\eps c}{2}. $$
Hence, by Lemma \ref{lemma:SM}, we obtain \eqref{eq:SMin} for $z \in \Omega$ which, combined with the bounds above, yields
$$ \sup_{z \in \Omega} \left| \frac{1}{n} \Bj^\mathrm{T} \left( zI - D_\In + \frac{1}{n} D_\In J_{n-s} \right)^{-1} D_\In \Bj \right| \leq 2 + MC + \frac{(2 + MC)^2}{\frac{c \eps}{2} } = O(1) $$
almost surely.  As before, \eqref{eq:SMin} also implies that the function in \eqref{eq:analytic} is analytic on $\Omega$.  The proof of the lemma is complete.  
\end{proof}

Let us dispatch the simplest case of Theorem \ref{thm:reduction}: when $s = 0$.  Indeed, if $s = 0$, then $D = D_\In$.  In this case, \eqref{eq:firstdeteq} and the invertibility of \eqref{eq:zIDmatrix} imply that $p_n$ has no critical points outside $\N_{\mu}(2 \eps)$, completing the proof.  Thus, for the remainder of the proof, we assume $s \geq 1$.  

We return to the block determinant in \eqref{eq:blockdet}.  By Lemma \ref{lemma:technical}, almost surely, for $n$ sufficiently large, the upper-left block is invertible for all $z \not\in \N_{\mu}(2 \eps)$.  Thus, by Lemma \ref{lemma:block}, we conclude that almost surely
\begin{align*}
	\frac{1}{n} z p_n'(z) = \det &\left(zI - D_\In + \frac{1}{n} D_\In J_{n-s} \right) \\
	&\qquad \times \det \left( zI - D_\Out + \frac{1}{n} D_\Out J_s - \frac{1}{n} D_\Out J_{s,n-s} G(z) \frac{1}{n} D_\In J_{n-s,s} \right) 
\end{align*}
for all $z \not\in \N_{\mu}(2 \eps)$, where
$$ G(z) := \left( zI - D_\In + \frac{1}{n} D_\In J_{n-s} \right)^{-1}. $$
In other words, the zeros of $p_n'$ outside of $\N_{\mu}(2\eps)$ (counted with multiplicity) are precisely the zeros of 
\begin{equation} \label{eq:exactzerosfirst}
	\det \left( zI - D_\Out + \frac{1}{n} D_\Out J_s - \frac{1}{n} D_\Out J_{s,n-s} G(z) \frac{1}{n} D_\In J_{n-s,s} \right) 
\end{equation}
outside of $\N_{\mu}(2\eps)$ (counted with multiplicity).  
Notice that this is the determinant of an $s \times s$ matrix, and $s \leq k_n = O(1)$.  We have thus reduced the problem of studying an $n \times n$ matrix to an $s \times s$ matrix.  This reduction greatly simplifies the forthcoming analysis.  Before we conclude the proof, we make one final observation: since $J_{s,n-s} = \Bj_s \Bj_{n-s}^\mathrm{T}$ and $J_{n-s,s} = \Bj_{n-s} \Bj_{s}^\mathrm{T}$, we can rewrite the determinant in \eqref{eq:exactzerosfirst} as
\begin{equation} \label{eq:exactzeros}
	\det \left( zI - D_\Out + \frac{1}{n} D_\Out J_s - \frac{1}{n^2} \left( \Bj_{n-s}^\mathrm{T} G(z) D_\In \Bj_{n-s} \right) D_\Out J_s \right). 
\end{equation}

We now conclude the proof of Theorems \ref{thm:reduction} and \ref{thm:reductiondist} separately.  Let us begin with Theorem \ref{thm:reduction}.  Indeed, under the assumptions of Theorem \ref{thm:reduction}, 
$$ \|D_{\Out}\| = \max\{|\xi_1^{(n)}|, \ldots, |\xi_s^{(n)}|\} =  O(1). $$  
(Recall that $\|D_\Out\|$ denotes the spectral norm of the matrix $D_\Out$.)  Thus, by Lemma \ref{lemma:detdiff} and Lemma \ref{lemma:technical}, we have, almost surely
\begin{align*}
	&\sup_{z \not\in \N_{\mu}(2\eps)} \left| \det \left( zI - D_\Out + \frac{1}{n} D_\Out J_s - \frac{1}{n^2} D_\Out J_{s,n-s} G(z) D_\In J_{n-s,s} \right) - \det(zI - D_\Out) \right| \\
	&\qquad\qquad \ll \frac{1}{n} \|D_\Out\| \|J_s\| + \|D_\Out\| \|J_s\| \sup_{z \not\in \N_{\mu}(2\eps)} \left| \frac{1}{n^2} \Bj_{n-s}^{\mathrm{T}}G(z) D_\In \Bj_{n-s} \right| \\
	&\qquad \qquad \ll \frac{1}{n}
\end{align*}
because $\|J_s\| = s \leq k_n = O(1)$.  Notice that the zeros of $\det(zI - D_{\Out})$ are precisely the values $\xi_1^{(n)}, \ldots, \xi_s^{(n)}$.  In view of Rouch\'{e}'s theorem (since both determinants are analytic outside $\N_{\mu}(2 \eps)$ due to Lemma \ref{lemma:technical}), we conclude that, almost surely, for $n$ sufficiently large, $p_n$ has exactly $s$ critical points outside $N_{\mu}(2\eps)$, and after correctly labeling the critical points, 
\begin{equation} \label{eq:improverouche}
	w_{l}(p_n) = \xi_{l}^{(n)} + o(1) 
\end{equation}
for each $1 \leq l \leq s$.  This completes the proof of Theorem \ref{thm:reduction}.

\begin{remark}
With a more careful application of Rouch\'{e}'s theorem, the error in \eqref{eq:improverouche} can be improved to \[w_l(p_n) = \xi_l^{(n)} + O(n^{-\tau}) \] for each $1 \leq l \leq s$, where $\tau > 0$ depends on $s$.  In addition, if the deterministic roots $\xi_l^{(n)}$, $1 \leq l \leq s$ satisfy some kind of separation criteria, this error term can be further improved.  We do not pursue these matters here.   
\end{remark}

We now turn to the proof of Theorem \ref{thm:reductiondist}.  Recall that, in this case, $s = 1$.  Thus, the matrix in \eqref{eq:exactzeros} is just a $1 \times 1$ matrix, and hence the zeros of $p_n'$ outside of $\N_{\mu}(2 \eps)$ are precisely the solutions of 
\begin{equation} \label{eq:distzeros}
	z - \xi_1^{(n)} + \frac{1}{n} \xi_1^{(n)} - \xi_1^{(n)} \frac{1}{n^2} \Bj^\mathrm{T} G(z) D_\In \Bj = 0 
\end{equation}
outside $\N_{\mu}(2 \eps)$.  By Lemma \ref{lemma:technical}, we have, almost surely, 
\begin{equation} \label{eq:rouchebnd}
	\sup_{z \not\in \N_{\mu}(2\eps)} \left| \left(z - \xi_1^{(n)} + \frac{1}{n} \xi_1^{(n)} - \xi_1^{(n)} \frac{1}{n^2} \Bj^\mathrm{T} G(z) D_\In \Bj \right) - \left(z - \xi_1^{(n)} \right) \right| \leq \frac{C}{n} |\xi_1^{(n)}| 
\end{equation}
for some constant $C > 0$.  Since both these terms are analytic outside $\N_{\mu}(2\eps)$ due to Lemma \ref{lemma:technical}, we can again apply Rouch\'{e}'s theorem.  However, since $\frac{C}{n} |\xi_1^{(n)}|$ does not necessarily converge to zero, we have to be slightly more careful.  Let $\Gamma_n$  be any simple closed contour outside $\N_{\mu}(2\eps)$ which satisfies $|z - \xi_1^{(n)}| > \frac{C}{n} |\xi_1^{(n)}|$ for all $z \in \Gamma_n$.  Then, by the estimate in \eqref{eq:rouchebnd}, Rouch\'{e}'s theorem implies that the number of solutions to \eqref{eq:distzeros} inside $\Gamma_n$ is the same as the number of zeros of $z - \xi_1^{(n)}$ inside $\Gamma_n$.  Hence, we conclude that almost surely, for $n$ sufficiently large, there is exactly one critical point of $p_n$ outside $\N_{\mu}(2\eps)$ and that critical point takes the value $\xi_1^{(n)}(1 + O(1/n))$.  The proof of Theorem \ref{thm:reductiondist} is complete.


\appendix

\section{Proof of Theorem \ref{thm:genkabluchko}} \label{sec:genkabluchko}

The proof of Theorem \ref{thm:genkabluchko} presented here is modeled after Kabluchko's proof of \cite[Theorem 1.1]{K}. We note that Theorem \ref{thm:genkabluchko} does not follow from the results in \cite{K}, and the notable difference between our proof and the one given in \cite{K} is that we must control the additional contribution coming from the deterministic triangular array.  
For convenience, we use $\mu_n$ and $\mu'_n$ to mean $\mu_{p_n}$ and $\mu_{p'_n}$, respectively and define 
\begin{equation} \label{eq:def:Xi}
	\Xi := \bigcup_{n=1}^\infty \left\{\xi_{l}^{(n)}: 1 \leq l \leq k_n \right\} 
\end{equation}
to be the collection of values present in the deterministic triangular array.  We let $\lambda$ represent Lebesgue measure on $\C$, and we denote the positive and negative parts of the real logarithm by
\[\log_{-}x := \begin{cases}\abs{\log{x}}, & 0 \leq x \leq 1,\\ 0,&x \geq 1,\end{cases}\quad\text{and}\quad \log_{+}x := \begin{cases} 0,& 0 \leq x \leq 1,\\ \log{x},&x \geq 1, \end{cases}\]
for $x \in [0, \infty)$.  We use the convention that $\log_{-}(0) := \infty$ so that $\log_{-}(\cdot)$ is a function taking values in the extended real line.  

We prove Theorem \ref{thm:genkabluchko} using the following result, which requires the deterministic array satisfy an additional assumption.  
\begin{theorem}
Under the same hypotheses as in Theorem \ref{thm:genkabluchko} and with the additional assumption that there is a set $E$ of Lebesgue measure zero for which $z \in \C \setminus E$ implies
\begin{equation}
\limsup_{n\to \infty}\frac{1}{n}\sum_{l=1}^{k_{n}}\log_{-}{\abs{z-\xi_l^{(n)}}} = 0,
\label{eqn:kabl:detCond}
\end{equation}
it follows that $\mu'_{n}$ converges weakly to $\mu$ in probability as $n \to \infty$. 
\label{thm:kabl:detCond}
\end{theorem}

Unfortunately, we cannot always guarantee that the deterministic array satisfies condition \eqref{eqn:kabl:detCond}.  To get around this issue, we will work on subsequences where the condition does hold; specifically, the proof of Theorem \ref{thm:genkabluchko} will require the following corollary of Theorem \ref{thm:kabl:detCond}.  
\begin{corollary}
Assume the same hypotheses as in Theorem \ref{thm:genkabluchko} and, in addition, suppose $\mu_{n_m}$ is a subsequence of $\mu_n$ for which there is a set $E\subset \C$ of zero Lebesgue measure such that $z \in \C \setminus E$ implies
\[
\limsup_{m\to \infty} \frac{1}{n_m}\sum_{l=1}^{k_{n_m}}\log_{-}{\abs{z-\xi_l^{(n_m)}}} = 0.
\]
Then $\mu'_{n_m}$ converges weakly to $\mu$ in probability as $n \to \infty$. 
\label{cor:kabl:subseq}
\end{corollary}
\begin{proof}
We show that $\mu_{n_m}$ is a subsequence of a new sequence of random measures (modified from $\mu_n$) for which condition \eqref{eqn:kabl:detCond} does hold. To this end, define the sequence $\tilde{k}_n$ by
\[\tilde{k}_n := \begin{cases}k_n, &\text{if $n = n_m$ for some $m \in {\mathbb{N}}$},\\0,&\text{otherwise},\end{cases}\]
and the random polynomial
\[
\tilde{p}_n(z) := \prod_{j=1}^{n-\tilde{k}_n}(z-X_j)\prod_{l=1}^{\tilde{k}_n}(z-\xi_l^{(n)}).
\]
Also let $\tilde{\mu}_n$ and $\tilde{\mu}'_n$ denote $\mu_{\tilde{p}_n}$ and $\mu_{\tilde{p}'_n}$, respectively. By construction, $\mu_{n_m}$ and $\mu'_{n_m}$ are subsequences of $\tilde{\mu}_n$ and $\tilde{\mu}'_n$, respectively. Now, $\tilde{k}_n = o(n)$, and for $z \in \C \setminus E$, 
\[
\limsup_{n\to \infty}\frac{1}{n}\sum_{l=1}^{\tilde{k}_{n}}\log_{-}{\abs{z-\xi_l^{(n)}}} =  \limsup_{m\to \infty} \frac{1}{n_m}\sum_{l=1}^{k_{n_m}}\log_{-}{\abs{z-\xi_l^{(n_m)}}} = 0.
\]
Thus, Theorem \ref{thm:kabl:detCond} implies that $\tilde{\mu}'_n$ converges weakly to $\mu$ in probability as $n \to \infty$. It follows that the subsequence $\mu'_{n_m}$ also converges to $\mu$ weakly in probability as $m \to \infty$.
\end{proof}

The following lemma will allow us to justify the use of Corollary \ref{cor:kabl:subseq}. 

\begin{lemma}
Let $\mu_n$ be a sequence of random probability measures on $\C$, and suppose $\mu$ is a deterministic probability measure on $\C$. Then, $\mu_n$ converges weakly to $\mu$ in probability if and only if each subsequence of $\mu_n$ contains a further subsequence that converges weakly to $\mu$ in probability.
\label{lma:kabl:subCriterion}
\end{lemma}
\begin{proof}
Observe that, for each bounded and continuous function $f:\C \to \R$, the sequence $\int_\C f\,d\mu_n$ is a sequence of complex-valued random variables whose subsequences are of the form $\int_\C f\,d\mu_{n_m}$, where $\mu_{n_m}$ is a subsequence of $\mu_n$. In addition, $\int f \,d\mu$ is a constant.  Thus, the claim follows by applying Theorem 2.6 on page 20 of \cite{PBill} to the random variables $\int_\C f\,d\mu_n$. 
\end{proof}

We now prove Theorem \ref{thm:genkabluchko} by way of Corollary \ref{cor:kabl:subseq} and Lemma \ref{lma:kabl:subCriterion}. The proof of Theorem \ref{thm:kabl:detCond} is delayed until Section \ref{sec:kabl:proofDetCond}.  Fix a subsequence $\mu'_{n_m}$ of $\mu'_n$.  We will show that there exists a further subsequence that converges weakly to $\mu$ in probability, which, by Lemma \ref{lma:kabl:subCriterion} would complete the proof of Theorem \ref{thm:genkabluchko}.  

Clearly, $\mu_{n_m}$ is a subsequence of $\mu_n$. If $\lambda$ denotes Lebesgue measure on $\C$, then Markov's inequality implies that, for any $\eps > 0$, 
\begin{align*}
\lambda\left(z \in \C : \frac{1}{n_m} \sum_{l=1}^{k_{n_m}}\log_{-}\abs{z-\xi_l^{(n_m)}} \geq \eps\right) &\leq \frac{1}{\eps}\int_\C\frac{1}{n_m} \sum_{l=1}^{k_{n_m}}\log_{-}\abs{z-\xi_l^{(n_m)}}\,d\lambda(z)\\
&= \frac{1}{\eps\cdot n_m} \sum_{l=1}^{k_{n_m}}\int_\C \log_{-}\abs{z-\xi_l^{(n_m)}}\,d\lambda(z)\\
&= \frac{k_{n_m}}{\eps\cdot n_m} \int_\C \log_{-}\abs{z}\,d\lambda(z).
\end{align*}
The last expression tends to zero as $m \to \infty$ by the local integrability of the logarithm and the fact that $k_n = o(n)$. Thus, the sequence of functions
\[z \mapsto \frac{1}{n_m} \sum_{l=1}^{k_{n_m}}\log_{-}\abs{z-\xi_l^{(n_m)}} \]
converges to zero in measure as $m \to \infty$.  Among other things, this implies that there exists a subsequence of this sequence that converges to zero for almost every $z\in \C$ (see, for instance, Theorem 2.30 on page 61 of \cite{F} for details). 
Let $\mu_{{n_m}_j}$ denote the corresponding subsequence of random measures. By Corollary \ref{cor:kabl:subseq}, we have that $\mu'_{{n_m}_j}$ converges weakly to $\mu$ in probability as $j \to \infty$, completing the proof.  


\subsection{Proof of Theorem \ref{thm:kabl:detCond}}\label{sec:kabl:proofDetCond}

It remains to prove Theorem \ref{thm:kabl:detCond}.  The proof presented here is modeled after the arguments given in \cite{K}. The case where $\mu$ is degenerate is straightforward to establish by computing $\mu'_{n}$ explicitly and directly verifying that $\abs{\int_\C f\,d\mu'_{n} - \int_\C f\,d\mu_{n}} \to 0$ almost surely as $n \to \infty$ for any bounded and continuous function $f:\C\to \R$. We now consider the case that $\mu$ is non-degenerate. 

The proof of Theorem \ref{thm:kabl:detCond} will reduce to studying the logarithmic derivative $L_n$ of $p_n$ defined by the formula
\[L_n(z) := \frac{p'_n(z)}{p_n(z)} = \sum_{j = 1}^{n-k_n}\frac{1}{z-X_j} + \sum_{l = 1}^{k_n}\frac{1}{z-\xi_l^{(n)}}.\] 
Specifically, Theorem \ref{thm:kabl:detCond} will follow from Lemma \ref{lma:kabl:intsmall} below. We also now state a related lemma (Lemma \ref{lma:kabl:Lnsmall}), which we will need later. Note that these two lemmas are very similar to \cite[Lemmas 2.1 and 2.2]{K}; however, neither lemma follows directly from the results in \cite{K} because of the deterministic contribution to $L_n$.

\begin{lemma}
Under the assumptions of Theorem \ref{thm:kabl:detCond}, there is a set $F \subset \C$ of Lebesgue measure zero such that if $z \in \C \setminus F$, then
\[\frac{1}{n}\log\abs{L_n(z)} \longrightarrow 0
\]
in probability as $n \to \infty$.  
\label{lma:kabl:Lnsmall}
\end{lemma}

\begin{lemma}
Under the assumptions of Theorem \ref{thm:kabl:detCond}, for any continuous, compactly supported function $\varphi:\C \to \R$, we have
\begin{equation}
\frac{1}{n}\int_\C\log\abs{L_n(z)}\varphi(z)\,d\lambda(z) \longrightarrow 0
\label{eqn:kabl:intsmall}
\end{equation}
in probability as $n \to \infty$.  (Recall that $\lambda$ denotes Lebesgue measure on $\mathbb{C}$.)
\label{lma:kabl:intsmall}
\end{lemma}

We now prove Theorem \ref{thm:kabl:detCond} assuming Lemma \ref{lma:kabl:intsmall}.  The key idea is the following formula (see, for instance, \cite[Section 2.4.1]{JHough}), which relates the integral in \eqref{eqn:kabl:intsmall} to the measures $\mu_n$ and $\mu'_n$.  For any polynomial $f$ that is not identically zero,
\[
\frac{1}{2\pi}\Delta\log\abs{f} = \sum_{z\in \C\,:\,f(z) = 0}\delta_z
\]
in the distributional sense, where each root in the sum is counted with multiplicity. In other words, for any compactly supported, smooth function $\varphi:\C\to \R$, we have
\[
\frac{1}{2\pi}\int_\C\log\abs{f(z)}\Delta\varphi(z)\,d\lambda(z) = \sum_{z\in \C\,:\,f(z) = 0}\varphi(z).
\]
  From this relationship we obtain that, for any smooth, compactly supported function $\varphi:\C \to \R$,  
\begin{equation*}
\frac{1}{n}\sum_{z \in \C\,:\, p_n'(z) = 0}\varphi(z) - \frac{1}{n}\sum_{z \in \C\,:\, p_n(z) = 0}\varphi(z) = \frac{1}{2\pi n}\int_\C \log\abs{L_n(z)} \Delta \varphi(z)\,d\lambda(z).
\end{equation*}
In view of Lemma \ref{lma:kabl:intsmall}, the integral on the right tends to zero in probability as $n\to \infty$. In addition, by the law of large numbers and the fact that $k_n = o(n)$, 
\[
\frac{1}{n}\sum_{z \in \C\,:\, p_n(z) = 0}\varphi(z) = \frac{1}{n}\sum_{j=1}^{n-k_n}\varphi(X_j) +\frac{1}{n}\sum_{l=1}^{k_n}\varphi(\xi_l^{(n)}) \longrightarrow \int_\C\varphi(z)\,d\mu(z)
\]
almost surely as $n \to \infty$. Hence, for any smooth, compactly supported function $\varphi:\C \to \R$
\begin{equation*}
\int_\C \varphi(z)\,d\mu'_n(z) = \frac{1}{n-1}\sum_{z \in \C\,:\, p_n'(z) = 0}\varphi(z) \longrightarrow \int_\C\varphi(z)\,d\mu(z)
\end{equation*}
in probability as $n \to \infty$. Since $\mu$ is a probability measure, we conclude from a simple approximation argument that $\mu'_n$ converges weakly to $\mu$ in probability.  This completes the proof of Theorem \ref{thm:kabl:detCond}.  


\subsubsection{Proof of Lemma \ref{lma:kabl:Lnsmall}}


We now turn our attention to proving Lemmas \ref{lma:kabl:Lnsmall} and \ref{lma:kabl:intsmall}.  We begin with Lemma \ref{lma:kabl:Lnsmall}, which we will need to prove Lemma \ref{lma:kabl:intsmall}.  
First, we construct the exceptional set $F$ described in Lemma \ref{lma:kabl:Lnsmall} from several smaller subsets. The first of these, $F_1$, contains points where $\mu$ misbehaves, while another, $F_2$, includes values too close to the deterministic array.  Define the set $F_1$ by 
\[
F_1 := \set{z \in \C: \int_\C\log^2_{-}\abs{z-y}\,d\mu(y) = \infty}.
\]
$F_1$ has Lebesgue measure zero since 
\begin{align*}
\int_\C \left(\int_\C\log^2_{-}\abs{z-y}\,d\mu(y)\right)\,d\lambda(z) &=  \int_\C \left(\int_\C\log^2_{-}\abs{z-y}\,d\lambda(z)\right)\,d\mu(y)\\
&= \int_\C \frac{\pi}{2}\,d\mu(y) = \frac{\pi}{2} < \infty 
\end{align*}
by the Fubini--Tonelli theorem.  

We now construct the subset $F_2$ by applying the Borel--Cantelli lemma.  Recall that the set $\Xi$, defined in \eqref{eq:def:Xi}, is at most countable, and hence $\lambda(\Xi) = 0$.  Thus, for a fixed $n \in {\mathbb{N}}$ and $1 \leq l \leq k_n$,
\begin{align*}
\lambda \left(z \in \C\setminus \Xi: \frac{1}{|z-\xi_l^{(n)}|} \geq e^{\sqrt{n}}\right) &= \lambda\left(z \in \C\setminus \Xi : \log_{-}|z-\xi_l^{(n)}| \geq \sqrt{n}\right) \\
&\leq\frac{1}{n^3}\int_\C \log_{-}^6|z-\xi_l^{(n)}|\,d\lambda(z) \\
&=\frac{C}{n^3}
\end{align*}
by Markov's inequality, where $C > 0$ is an absolute constant equal to the integral of $\log^6_{-}\abs{\cdot}$ over $\C$. Thus, we obtain 
\[
\sum_{n=1}^\infty \sum_{l=1}^{k_n}\lambda \left(z \in \C\setminus \Xi: \frac{1}{|z-\xi_l^{(n)}|} \geq e^{\sqrt{n}}\right) \leq \sum_{n=1}^\infty \sum_{l=1}^{k_n} \frac{C}{n^3} = \sum_{n=1}^\infty\frac{Ck_n}{n^3} < \infty
\]
since $k(n) = o(n)$. It follows by the Borel--Cantelli lemma and the fact that $\Xi$ is countable that there exists a set $F_2 \supset \Xi$ of Lebesgue measure zero such that, for every $z \in \C \setminus F_2$, $|z-\xi_l^{(n)}|^{-1} < e^{\sqrt{n}}$ for all but finitely many pairs $(n,l)$. We conclude that, for $z\in \C \setminus F_2$,  
\begin{equation}
\sum_{l=1}^{k_n}\frac{1}{|z-\xi_l^{(n)}|} = O_z(e^{2\sqrt{n}}),
\label{eqn:kabl:detCond2}
\end{equation}
where the asymptotic notation $O_z(\cdot)$ means the implicit constant is allowed to depend on $z$.  

If we define $F$ to be
$
F:= E \cup F_1\cup F_2, 
$
then $F$ has Lebesgue measure zero and, as we shall see, satisfies the requirements of Lemma \ref{lma:kabl:Lnsmall}. (Recall the definition of $E$ from the statement of Theorem \ref{thm:kabl:detCond} above.) Notice that $F$ contains the atoms of $\mu$ and the values in the deterministic triangular array.


\begin{lemma}
For every $z \in \C \setminus F$, \[\limsup_{n\to \infty} \frac{1}{n}\log\abs{L_n(z)} \leq 0\] almost surely.
\label{lma:kabl:logbelowzero}
\end{lemma}
\begin{proof}
Fix $z \in \C \setminus F$, and let $\eps > 0$ be given. By Markov's inequality, for any $n \in {\mathbb{N}}$, we have
\begin{align*}
\P\left(\frac{1}{\abs{z- X_n}} \geq e^{\eps n}\right) &= \P\left(\log_{-}\abs{z- X_n} \geq \eps n\right)\\
&\leq \frac{\E\left[\log^2_{-}\abs{z- X_n}\right]}{\eps^2 n^2}\\
&= \frac{1}{\eps^2 n^2} \int_\C\log^2_{-}\abs{z-y}\,d\mu(y)\\
&= \frac{C_1}{\eps^2 n^2},
\end{align*}
for a non-negative constant $C_1$ since $z \notin F_1$. Hence, 
\[
\sum_{n=1}^\infty \P\left(\frac{1}{\abs{z- X_n}} \geq e^{\eps n}\right) < \infty, 
\]
so the Borel--Cantelli lemma applies. In particular, almost surely $\frac{1}{|z-X_n|} < e^{\eps n}$ for all but finitely many $n$. Furthermore, $z$ is not an atom of $\mu$, so we have almost surely that, for all $n$,
\[
\abs{L_n(z)} \leq W + (n-k_n)e^{\eps n} + \sum_{l=1}^{k_n}\frac{1}{|z-\xi_{l}^{(n)}|}, 
\]
where $W$ is an almost surely finite random variable. Now, since $z \in \C \setminus F_2$, the bound in \eqref{eqn:kabl:detCond2} implies that, for $n$ sufficiently large, 
\[
\abs{L_n(z)} \leq W + ne^{\eps n} + C_2e^{2\sqrt{n}} \leq e^{2\eps  n}
\]
for a positive constant $C_2$ (depending on $z$).  It follows that \[\limsup_{n\to \infty}\frac{1}{n}\log\abs{L_n(z)} \leq 2\eps\]
almost surely.  Since $\eps > 0$ was arbitrary, the proof is complete.  
\end{proof}


The reverse inequality in Lemma \ref{lma:kabl:Lnsmall} requires an anti-concentration result that can be found, for example, in \cite[Theorem 2.22 on page 76]{P}. Before stating the lemma, we define the L\'{e}vy concentration function of a complex-valued random variable.  

\begin{definition}[L\'{e}vy concentration function]
Let $Z$ be a complex-valued random variable.  The \emph{L\'{e}vy concentration function} of $Z$ is defined as
\[ \Le(Z,t) := \sup_{u \in \C} \P\left(\abs{Z-u}\leq t\right) \]
for all $t \geq 0$.  
\end{definition}

The L\'{e}vy concentration function bounds the \emph{small ball probabilities} for $Z$, which are the probabilities that $Z$ falls in a ball of radius $t$.  

\begin{lemma}[Anti-concentration estimate]
Suppose that $Z_1,\ldots, Z_n$ are iid, non-degenerate, complex-valued random variables. Then, there is a positive constant $C$ (depending only on the distribution of $Z_1$), so that, for any $t \geq 0$,
\begin{equation}
\Le\left(Z_1 + \cdots + Z_n, t\right) \leq C\frac{1 + t}{\sqrt{n}}
\label{eqn:kabl:anticoncentration}
\end{equation}
for all $n \geq 1$.
\label{lma:anticoncentration}
\end{lemma}

\begin{proof}
Theorem 2.22 on page 76 in \cite{P} implies that equation \eqref{eqn:kabl:anticoncentration} holds when  $Z_1, \ldots, Z_n$ are iid real-valued random variables and the supremum in the concentration function is taken over real numbers (see also \cite[Corollary 6.8]{OT} for a more general version of this inequality). We extend this to the complex case in the following way. By assumption, $Z_1, \ldots, Z_n$ are iid and non-degenerate, so at least one of the real-valued random variables $\Re(Z_1)$ or $\Im(Z_1)$ is non-degenerate.  Without loss of generality, assume $\Re(Z_1)$ is non-degenerate. Then
\begin{align*}
\Le(Z_1+\cdots+ Z_n, t) &= \sup_{u \in \C}\P\left(\abs{Z_1+\cdots + Z_n-u} \leq t\right)\\
&\leq \sup_{u \in \C}\P\left(\abs{\Re(Z_1) + \cdots + \Re(Z_n) - \Re(u)}\leq t\right)\\
&= \sup_{u \in \R}\P\left(\abs{\Re(Z_1) + \cdots + \Re(Z_n) - u}\leq t\right).
\end{align*}
The last expression is bounded by $C\frac{1 + t}{\sqrt{n}}$, for some constant $C$ that depends only on the distribution of $\Re(Z_1)$ by the previously mentioned result in \cite{P}. A nearly identical argument applies if $\Re(Z_1)$ is degenerate and $\Im(Z_1)$ is non-degenerate.
\end{proof}

\begin{lemma} For every $z \in \C \setminus F$ and every $\eps > 0$,
\[
\lim_{n\to \infty} \P \bigg[\frac{1}{n}\log\abs{L_n(z)} \leq -\eps\bigg] = 0.
\]
\label{lma:kabl:logaboveneg}
\end{lemma}

\begin{proof}
Since $k_n = o(n)$, we assume $n$ is sufficiently large so that $k_n < n$.  
Fix $z \in \C\setminus F$, and let $\eps > 0$ be given. Since $\mu$ is non-degenerate and $z$ is not an atom of $\mu$, it follows that $\frac{1}{z-X_1}, \frac{1}{z-X_2}, \ldots$ are iid, non-degenerate, complex-valued random variables satisfying the hypotheses of Lemma \ref{lma:anticoncentration}. By absorbing the contribution of $\sum_{l=1}^{k_n} (z - \xi_l^{(n)})^{-1}$ into the complex number $u$ in the definition of the concentration function, we conclude from Lemma \ref{lma:anticoncentration} that
\[
\P\left(\abs{L_n(z)} \leq e^{-\eps n}\right) \leq \Le\left(\sum_{j=1}^{n-k_n}\frac{1}{z-X_j},\ e^{-\eps n}\right) \leq C\frac{1 + e^{-\eps n}}{\sqrt{n-k_n}}
\]
for a positive constant $C$ depending only on the distribution of $\frac{1}{z - X_1}$. As $n \to \infty$, the right-hand side goes to zero (since $k_n = o(n)$), which completes the proof.
\end{proof}

Together, Lemmas \ref{lma:kabl:logbelowzero} and \ref{lma:kabl:logaboveneg} establish Lemma \ref{lma:kabl:Lnsmall}.


\subsubsection{Proof of Lemma \ref{lma:kabl:intsmall}}

In this section, we prove Lemma \ref{lma:kabl:intsmall} by way of the following dominated convergence result due to Tau and Vu \cite{TaoVu}.

\begin{lemma}[Tao--Vu; Lemma 3.1 in \cite{TaoVu}]
Let $(X,\mathcal A, \nu)$ be a finite measure space, and let $f_1, f_2, \ldots: X\to \R$ be random functions which are defined over a probability space $(\Omega,\BB,\P)$ and are jointly measurable with respect to $\mathcal A\otimes\BB$. Assume that
\begin{enumerate}[(i)]
\item for $\nu$-a.e.\ $x \in X$ we have $f_n(x) \to 0$ in probability, as $n \to \infty$,
\item for some $\delta > 0$, the sequence $\int_X\abs{f_n(x)}^{1+\delta}\,d\nu(x)$ is tight.
\end{enumerate}
Then $\int_X f_n(x)\,d\nu(x)$ converges in probability to $0$.
\label{lma:kabl:tauvu}
\end{lemma}

In order to prove Lemma \ref{lma:kabl:intsmall}, we will apply Lemma \ref{lma:kabl:tauvu} to the random functions $f_n(z) := \frac{1}{n}(\log\abs{L_n(z)})\varphi(z)$, where $\varphi$ is a continuous function with compact support. Lemma \ref{lma:kabl:Lnsmall} establishes the first condition, and the tightness condition (with $\delta = 1$) follows from the next lemma. For the remainder of the paper, we let 
\[
\D_R := \set{z \in \C: \abs{z} < R}
\]
denote the open disk of radius $R > 0$ centered about the origin.  Fix $r>0$ such that the support of $\varphi$ is contained in the open disk $\D_r$.  We will occasionally use $\ind_{\D_r}$ to denote the indicator function of the set $\D_r$.

\begin{lemma}
The sequence $\frac{1}{n^2}\int_{\D_r}\log^2\abs{L_n(z)}\,d\lambda(z)$ is tight.
\label{lma:kabl:intTight}
\end{lemma}

In view of Lemma \ref{lma:kabl:tauvu}, the proof of Lemma \ref{lma:kabl:intsmall} reduces to establishing Lemma \ref{lma:kabl:intTight}.  We bound the integral in Lemma \ref{lma:kabl:intTight} by employing the Poisson--Jensen formula as in \cite{K}.  In order to do so, we will need a uniform bound on $\abs{L_n(z)}$ for $z$ of certain magnitudes, which is the content of the following lemma.  

\begin{lemma}
There is an exceptional set $G \subset (0,\infty)$ of Lebesgue measure zero such that, for any $R \in (0,\infty) \setminus G$, we have
\begin{equation}
\limsup_{n\to \infty} \frac{1}{n}\log \sup_{\abs{z} = R}\abs{L_n(z)} \leq 0
\label{eqn:kabl:unifzBound}
\end{equation}
almost surely.  
\label{lma:kabl:unifzBound}
\end{lemma}
\begin{proof}
The proof is similar in spirit to that of Lemma \ref{lma:kabl:logbelowzero}.  We first claim that 
\begin{equation}
\sup_{\abs{z} = R}\frac{1}{\abs{z-X}} \geq K\ \Longleftrightarrow \abs{\abs{X} - R} \leq \frac{1}{K}, 
\label{eqn:kabl:supEquiv} 
\end{equation}
for any $X \in \C$, $R \in (0,\infty) \setminus \{\abs{X}\}$, and $K>0$. This equivalence will allow us to employ the method of Lemma \ref{lma:kabl:logbelowzero} and control the behavior of $\log_{-}{\abs{\abs{X_n} - R}}$. To establish the forward direction of \eqref{eqn:kabl:supEquiv}, observe that
\[0 < \abs{\abs{X} - R}  = \abs{\abs{X} - \abs{z}} \leq \abs{X - z}\] for any $z$ satisfying $\abs{z} = R$. Hence, 
$\sup_{\abs{z} = R}{\abs{X - z}}^{-1} \geq K$ implies $\abs{\abs{X} - R} \leq K^{-1}$. On the other hand, if $\abs{\abs{X} - R} \leq K^{-1}$, write $X = \rho e^{i\theta}$ in polar coordinates, and note that $z^* := Re^{i\theta}$ has modulus $R\neq \rho$ and satisfies 
\[
0 < \abs{z^* - X} = \abs{R - \rho} = \abs{\abs{X}- R} \leq \frac{1}{K}.
\]
The fact that $\sup_{\abs{z} = R}{\abs{X - z}}^{-1} \geq K$ follows.

We are ready to construct $G$ from two exceptional sets $G_1$ and $G_2$. Define 
\[
G_1 := \set{R \in (0,\infty): \int_\C\log_{-}^2\abs{\abs{y}-R}\,d\mu(y) = \infty}.
\]
It follows from the Fubini--Tonelli theorem that $G_1$ has Lebesgue measure zero since
\[
\int_\R\int_\C \log_{-}^2\abs{\abs{y}-R}\,d\mu(y)\,dR = \int_\C\int_\R\log_{-}^2\abs{\abs{y}-R}\,dR\,d\mu(y) = \int_\C 2 \,d\mu(y) =2 < \infty.
\]

We now construct $G_2$. Let $\lambda_\R$ denote Lebesgue measure on the real line, and let 
\[ \Xi_\R := \bigcup_{n=1}^\infty \left\{|\xi_l^{(n)}|: 1 \leq l \leq k_n \right\}. \]  
Clearly, $\lambda_\R(\Xi_\R) = 0$.  Equivalence \eqref{eqn:kabl:supEquiv} and Markov's inequality imply that for a fixed $n \in {\mathbb{N}}$ and $1 \leq l \leq k_n$,
\begin{align*}
\lambda_\R &\left( R\in (0,\infty)\setminus \Xi_\R : \sup_{\abs{z}=R}\frac{1}{|z-\xi_l^{(n)}|} \geq e^{\sqrt{n}}  \right) \\
&\qquad\qquad\qquad= \lambda_\R\left(R\in (0,\infty)\setminus \Xi_\R:\log_{-}||\xi_l^{(n)}|-R|\geq \sqrt{n}\right)\\
&\qquad\qquad\qquad\leq\frac{1}{n^3}\int_{[0,\infty)}\log_{-}^6||\xi_l^{(n)}|-R|\,dR \\
&\qquad\qquad\qquad\leq\frac{1}{n^3}\int_\R\log_{-}^6\abs{R}\,dR\\
&\qquad\qquad\qquad = \frac{C}{n^3},
\end{align*}
where $C > 0$ is an absolute constant.  
It follows that
\[
\sum_{n=1}^\infty\sum_{l=1}^{k_n} \lambda_\R\left(R \in (0,\infty) \setminus \Xi_\R: \sup_{\abs{z}=R}\frac{1}{|z-\xi_l^{(n)}|} \geq e^{\sqrt{n}}\right) \leq \sum_{n=1}^{\infty} \frac{Ck_n}{n^3} < \infty,
\]
so the Borel--Cantelli lemma and the countability of $\Xi_\R$ show that outside of a set $G_2 \supset \Xi_\R$ of Lebesgue measure zero, 
\[
\sup_{\abs{z}=R}\frac{1}{|z-\xi_l^{(n)}|} < e^{\sqrt{n}}
\]
for all but finitely many pairs $(n,l)$. Hence, for $R \in (0,\infty) \setminus G_2$, 
\begin{equation}
\sum_{l=1}^{k_n}\sup_{\abs{z}=R}\frac{1}{|z-\xi_l^{(n)}|} < C_R + k_ne^{\sqrt{n}} = O_R(e^{2\sqrt{n}}),
\label{eqn:kabl:supDetBd}
\end{equation}
where $C_R$ is a positive constant depending on $R$. (Note that since $\Xi_\R \subset G_2$, $\sup_{\abs{z}=R}|z-\xi_l^{(n)}|^{-1} < \infty$ for each pair $(n,l)$). If we define $G = G_1 \cup G_2$, then, $G\subset (0, \infty)$ has Lebesgue measure zero, and for $R \in (0,\infty) \setminus G$, we have that, for any $n\in {\mathbb{N}}$ and any $\eps > 0$,
\begin{align*}
\P\left(\sup_{\abs{z}=R}\frac{1}{\abs{z-X_n}}\geq e^{\eps n}\right) &= \P\left(\log_{-}\abs{\abs{X_n}-R} \geq \eps n \right) \\
&\leq \frac{1}{\eps^2n^2}\E[\log^2_{-}\abs{\abs{X_n}-R}] \\
&= \frac{C_R'}{n^2},
\end{align*}
where we used \eqref{eqn:kabl:supEquiv} in the first step and Markov's inequality in the second.  Here, $C_R'$ is a positive constant depending only on $R$ and $\mu$. By the Borel--Cantelli lemma, it follows that almost surely, $\sup_{\abs{z}=R}\frac{1}{\abs{z-X_n}}< e^{\eps n}$ for all but finitely many $n$. This guarantees that for $R \in (0,\infty)\setminus G$, there is an almost surely bounded, real-valued random variable $W_R$ for which 
\[
\sup_{\abs{z} = R}\abs{L_n(z)} \leq W_R + (n-k_n)e^{\eps n} + \sum_{l=1}^{k_n}\sup_{\abs{z}=R}\frac{1}{|z-\xi_l^{(n)}|} \leq e^{2\eps n}
\]
almost surely. (Note that $\P(\abs{X_n} = R) = 0$ for all $R \in (0,\infty) \setminus G$ by the definition of the set $G_1$.) The last inequality holds for all sufficiently large $n$ by \eqref{eqn:kabl:supDetBd}. As $\eps > 0$ was arbitrary, \eqref{eqn:kabl:unifzBound} now follows.  
\end{proof}


We now use the Poisson--Jensen formula to re-write $\log\abs{L_n(z)}$. For any $R > r$ and $n \in {\mathbb{N}}$, let
\[
y_1^{(n)}, \ldots, y_{s_n}^{(n)}\quad\text{and}\quad w^{(n)}_{1}, \ldots, w_{t_n}^{(n)}
\]
be the roots and critical points, respectively, of $p_n$ that are located in the open disk $\D_{R}$. The Poisson--Jensen formula (see, for example, \cite[Chapter II.8]{Ma}) implies that for any $z \in \D_{R}$ which is not a zero or pole of $L_n$, 
\begin{equation}
\log\abs{L_n(z)} = I_n(z;R) + \sum_{t=1}^{t_n}\log\abs{\frac{R\left(z-w_{t}^{(n)}\right)}{R^2 - \overline{w_{t}^{(n)}}\,z}} - \sum_{s=1}^{s_n}\log\abs{\frac{R\left(z-y_s^{(n)}\right)}{R^2 - \overline{y_s^{(n)}}\,z}},
\label{eqn:kabl:pjen}
\end{equation}
where
\begin{equation*}
I_n(z;R) := \frac{1}{2\pi} \int_0^{2\pi} \log\abs{L_n(Re^{i\theta})}P_{R}(\abs{z},\theta - \arg{z})\,d\theta,
\end{equation*}
and $P_{R}$ denotes the Poisson kernel
\begin{equation}
P_{R}(\rho,\alpha) := \frac{R^2 - \rho^2}{R^2 + \rho^2 - 2R\rho\cos \alpha},\quad \rho \in [0,R],\ \alpha \in [0,2\pi].
\label{eqn:kabl:pker}
\end{equation}

\begin{lemma} 
There exists an $R \geq \max\set{1,3r}$ such that
\begin{equation}
\limsup_{n\to \infty}\frac{1}{n}\sup_{z\in \D_r}I_n(z;R)\leq 0
\label{eqn:kabl:supIbelow0}
\end{equation}
almost surely.
\label{lma:kabl:bdpint}
\end{lemma}
\begin{proof}
Fix $z \in \D_r$. Then, for any $\alpha \in [0,2\pi]$ and $R \geq 3r$, we have
\begin{equation}
P_{R}(\abs{z},\alpha) = \frac{R^2 - \abs{z}^2}{R^2 + \abs{z}^2 - 2R\abs{z}\cos \alpha}\leq \frac{(R + \abs{z})(R-\abs{z})}{R^2 + \abs{z}^2 - 2R\abs{z}} = \frac{R+\abs{z}}{R-\abs{z}} \leq  2.
\label{eqn:kabl:bdpkerabove}
\end{equation}
The last inequality follows from the fact that $\abs{z} \leq r$ and from the equivalence
\[
\frac{R+r}{R-r} \leq 2 \quad \Longleftrightarrow \quad R \geq 3r,
\]
which holds for all $R > r > 0$.  
Consequently, for any $z \in \D_r$ and $R\geq 3r$,
\begin{align*}
\frac{1}{n}I_n(z;R) &\leq \frac{1}{2\pi}\int_0^{2\pi}\frac{1}{n}\log\abs{L_n(Re^{i\theta})}\cdot 2\,d\theta\\
&\leq \frac{1}{\pi}\int_0^{2\pi}\frac{1}{n}\log\sup_{\abs{w} = R}\abs{L_n(w)}\,d\theta\\
&= \frac{2}{n}\log\sup_{\abs{w} = R}\abs{L_n(w)}.
\end{align*}
Therefore, we obtain
\begin{equation}
\limsup_{n\to \infty}\frac{1}{n}\sup_{z\in \D_r}I_n(z;R) \leq \limsup_{n\to \infty}\frac{2}{n}\log\sup_{\abs{w} = R}\abs{L_n(w)}.
\label{eqn:kabl:bdsupIabove}
\end{equation}
The desired result now follows by applying Lemma \ref{lma:kabl:unifzBound} to \eqref{eqn:kabl:bdsupIabove}. In particular, since the exceptional set $G \subset (0,\infty)$ of Lemma \ref{lma:kabl:unifzBound} has measure zero, we can choose $R \geq \max\set{1,3r}$ so that \eqref{eqn:kabl:supIbelow0} holds almost surely.  
\end{proof}



Next, we show that $I_n(z;R)$ is bounded below uniformly for $z \in \D_r$. We assume that $0 \notin F$, and we first consider the case when $z=0$. There is no loss of generality in assuming $0 \notin F$, for if $0 \in F$, we can choose a different point $c \notin F$ and prove Theorem \ref{thm:kabl:detCond} for the random variables $\widetilde{X}_j := X_j - c$ and the deterministic array $\tilde{\xi}_l^{(n)} := \xi_l^{(n)} - c$.  This follows since the translation of the roots of $p_n$ by $c$ simply translates the critical points by $c$ (see Proposition \ref{prop:translate}).

\begin{lemma}
Suppose $0 \notin F$. Let $R \geq \max\set{1,3r}$ be the value from Lemma \ref{lma:kabl:bdpint}.  Then there exists a non-negative constant $A$ such that
\[
\lim_{n\to \infty}\P\left(\frac{1}{n}I_n(0; R) \leq -A \right) = 0.
\]
\label{lma:kabl:In0bd}
\end{lemma}
\begin{proof}
Since $0 \notin F$, we have $p_n(0) \neq 0$ almost surely; in other words, $0$ is almost surely not a pole of $L_n$. Furthermore, by Lemma \ref{lma:kabl:logaboveneg}, it follows that $0$ is not a zero of $L_n$ with probability $1-o(1)$. Consequently, on the same event, the Poisson--Jensen formula \eqref{eqn:kabl:pjen} applies to $z = 0 \in \D_r$, and we obtain 
\begin{gather}
\begin{aligned}
\frac{1}{n}I_n(0;R) &= \frac{1}{n}\log\abs{L_n(0)} - \frac{1}{n}\sum_{t=1}^{t_n}\log\abs{\frac{w_{t}^{(n)}}{R}} +  \frac{1}{n}\sum_{s=1}^{s_n}\log\abs{\frac{y_s^{(n)}}{R}}\\
&\geq\frac{1}{n}\log\abs{L_n(0)} +  \frac{1}{n}\sum_{s=1}^{s_n}\log\abs{\frac{y_s^{(n)}}{R}}.
\end{aligned}
\label{step:kabl:pjenat00}
\end{gather}
The inequality comes from eliminating 
\[
\frac{1}{n}\sum_{t=1}^{t_n}\log\abs{\frac{w_{t}^{(n)}}{R}} \leq 0.
\]
We bound the remaining two terms in probability. A bound for the first term follows from Lemma \ref{lma:kabl:logaboveneg}.  
It remains to find a lower bound (in probability) for the last term in \eqref{step:kabl:pjenat00}. Let 
\[
x_{1}^{(n)}, \ldots, x_{u_{n}}^{(n)} \quad\text{and}\quad \zeta_{1}^{(n)}, \ldots, \zeta_{v_{n}}^{(n)}
\]
be the random and deterministic roots, respectively, of $p_n$ that are contained in $\D_{R}$. (Note that $u_n + v_n = s_n$.) The law of large numbers implies that 
\[
\frac{1}{n-k_n}\sum_{u=1}^{u_{n}}\log\abs{\frac{x_{u}^{(n)}}{R}} = \frac{1}{n-k_n}\sum_{j=1}^{n-k_n}\log\abs{\frac{X_j}{R}}\ind_{\D_{R}}(X_j) \longrightarrow -\E\log_{-}{\abs{\frac{X_1}{R}}}
\]
almost surely as $n \to \infty$.  The expectation on the right-hand side is finite since $\E\log_{-}\abs{X_1} < \infty$ due to the assumption $0 \notin F$ and by the bounds 
\begin{align*}
-\E\left[\log_{-}\abs{\frac{X_1}{R}}\right] &= -\E\left[\log_{-}\abs{\frac{X_1}{R}} - \log_{-}\abs{X_1} + \log_{-}\abs{X_1}\right]\\
&\geq-\E\left[\log_{-}\left(\frac{1}{R}\right)\right] - \E\left[\log_{-}\abs{X_1}\right] \\
&\geq-\log(R) - \E\left[\log_{-}\abs{X_1}\right],
\end{align*}
which follow from the fact that $R \geq 1$.  Since $k_n=o(n)$, it follows that
\[
\frac{1}{n}\sum_{u=1}^{u_{n}}\log\abs{\frac{x_{u}^{(n)}}{R}} \longrightarrow -\E\left[\log_{-}\abs{\frac{X_1}{R}}\right] \geq -\log(R) - \E\log_{-}\abs{X_1}
\]
almost surely as $n \to \infty$, and as a consequence, we have almost surely
\begin{equation}
\liminf_{n\to \infty}\frac{1}{n}\sum_{u=1}^{u_{n}}\log\abs{\frac{x_{u}^{(n)}}{R}} \geq -A_1,
\label{eqn:kabl:pjenat02}
\end{equation}
for some non-negative constant $A_1$ (depending on $R$). Similarly, as $R \geq 1$, we have
\begin{align*}
0 \geq \frac{1}{n}\sum_{v=1}^{v_{n}}\log\abs{\frac{\zeta_{v}^{(n)}}{R}} &= \frac{1}{n}\sum_{l=1}^{k_n}\log\abs{\frac{\xi_l^{(n)}}{R}}\ind_{\D_{R}}\left(\xi_l^{(n)}\right)\\
&= -\frac{1}{n}\sum_{l=1}^{k_n}\log_{-}\abs{\frac{\xi_l^{(n)}}{R}}\\
&= -\frac{1}{n}\sum_{l=1}^{k_n}\left(\log_{-}\abs{\frac{\xi_l^{(n)}}{R}}-\log_{-}\abs{\xi_l^{(n)}}\right) - \frac{1}{n}\sum_{l=1}^{k_n}\log_{-}\abs{\xi_l^{(n)}}\\
&\geq -\frac{1}{n}\sum_{l=1}^{k_n}\log_{-}\abs{\frac{1}{R}} - \frac{1}{n}\sum_{l=1}^{k_n}\log_{-}\abs{\xi_l^{(n)}} \\
&= -\frac{k_n}{n}\log(R) - \frac{1}{n}\sum_{l=1}^{k_n}\log_{-}\abs{\xi_l^{(n)}}.
\end{align*}
By condition \eqref{eqn:kabl:detCond} and the fact that $k_n = o(n)$, we obtain 
\begin{equation}
\lim_{n\to \infty}\frac{1}{n}\sum_{v=1}^{v_{n}}\log\abs{\frac{\zeta_{v}^{(n)}}{R}}=0.
\label{eqn:kabl:pjenat03}
\end{equation}
(Recall that $0 \notin F$, and hence $0 \notin E$.) Together, \eqref{eqn:kabl:pjenat02} and  \eqref{eqn:kabl:pjenat03} imply the desired conclusion.
\end{proof}


\begin{lemma}
Suppose $0 \notin F$.  Let $R \geq \max\set{1,3r}$ be the constant from Lemma \ref{lma:kabl:bdpint}.  Then there exists a non-negative constant $B$ such that 
\[\lim_{n\to \infty}\P\left(\frac{1}{n}\inf_{z\in \D_r} I_n(z; R) \leq -B\right) = 0.\]
\label{lma:kabl:bdpintbelow}
\end{lemma}
\begin{proof}
The proof presented here closely follows the arguments in \cite{K}. For simplicity, define
\[
q^+_n(\theta) := \frac{1}{n}\log_{+}\abs{L_n(Re^{i\theta})}\quad\text{and}\quad q^-_n(\theta) := \frac{1}{n}\log_{-}\abs{L_n(Re^{i\theta})}
\]
for $\theta \in [0,2\pi]$. By the definition of the Poisson kernel \eqref{eqn:kabl:pker} and reasoning similar to that used to derive the bounds in \eqref{eqn:kabl:bdpkerabove}, we have 
\[ \frac{1}{2} \leq P_R(\abs{z},\theta) \leq 2 \] 
for all $z \in \D_r$ and $\theta \in [0,2\pi]$. Notice that $P_R(0,\theta) =1$ for all $\theta \in [0,2 \pi]$, so we have
\[
\frac{2\pi}{n} I_n(0;R) = \int_0^{2\pi}q^+_n(\theta)\,d\theta - \int_0^{2\pi}q^-_n(\theta)\,d\theta.
\]
It follows that, for any $n \in {\mathbb{N}}$ and any $z \in \D_r$,
\begin{align*}
\frac{2\pi}{n}I_n(z;R) &= \int_0^{2\pi}q^+_n(\theta)P_R(\abs{z}, \theta - \arg z)\,d\theta - \int_0^{2\pi}q^-_n(\theta)P_R(\abs{z}, \theta - \arg z)\,d\theta\\
&\geq \frac{1}{2}\int_0^{2\pi}q^+_n(\theta)\,d\theta - 2\int_0^{2\pi}q^-_n(\theta)\,d\theta\\
&= \left(\frac{1}{2} - 2\right)\int_0^{2\pi}q^+_n(\theta)\,d\theta + 2\left(\int_0^{2\pi}q^+_n(\theta)\,d\theta - \int_0^{2\pi}q^-_n(\theta)\,d\theta\right)\\
&= -\frac{3}{2}\int_0^{2\pi}q^+_n(\theta)\,d\theta + \frac{4\pi}{n}I_n(0;R).
\end{align*}
In the case where $q^+_n(\theta) = 0$ for all $\theta \in [0,2\pi]$, we obtain the bound 
\[
\frac{2\pi}{n}I_n(z;R) \geq \frac{4\pi}{n}I_n(0;R).
\]
Otherwise, \[q_n^+(\theta) \leq \frac{1}{n}\log \sup_{\abs{z} = R}\abs{L_n(z)} \] for all $\theta \in [0,2\pi]$, and continuing from above,
\begin{align*}
\frac{2\pi}{n}I_n(z;R) &\geq \frac{4\pi}{n}I_n(0;R)-\frac{3}{2}\int_0^{2\pi}\frac{1}{n}\log \sup_{\abs{z} = R}\abs{L_n(z)}\,d\theta  \\
&= \frac{4\pi}{n}I_n(0;R)-\frac{3\pi}{n}\log \sup_{\abs{z} = R}\abs{L_n(z)}.
\end{align*}
In either case, taking the infimum over all $z \in \D_r$ and applying the results of Lemmas
 \ref{lma:kabl:unifzBound} and \ref{lma:kabl:In0bd} gives the desired conclusion.
\end{proof}

We complete the proof of Lemma \ref{lma:kabl:intTight} by applying Lemma \ref{lma:kabl:bdpint} and Lemma \ref{lma:kabl:bdpintbelow} to \eqref{eqn:kabl:pjen}. Let $R \geq \max\set{1,3r}$ be as in Lemma \ref{lma:kabl:bdpint}.   From \eqref{eqn:kabl:pjen}, we apply the Cauchy--Schwarz inequality twice to obtain 
\begin{gather}
\begin{aligned}
\frac{1}{n^2}\log^2\abs{L_n(z)} &\leq \frac{3}{n^2}I_n^2(z;R) + \frac{3t_n}{n^2}\sum_{t=1}^{t_n}\log^2\abs{\frac{R\left(z-w_{t}^{(n)}\right)}{R^2 - \overline{w_{t}^{(n)}}\,z}} \\
&\qquad\qquad+ \frac{3s_n}{n^2}\sum_{s=1}^{s_n}\log^2\abs{\frac{R\left(z-y_s^{(n)}\right)}{R^2 - \overline{y_s^{(n)}}\,z}},
\end{aligned}
\label{eqn:kabl:finalIneq}
\end{gather}
for $z \in \D_R$ that is not a zero or pole of $L_n$. Since there are finitely many zeros and poles of $L_n$ for a fixed $n$ and a fixed realization of $L_n$, \eqref{eqn:kabl:finalIneq} implies
\begin{gather}
\begin{aligned}
\frac{1}{n^2}\int_{\D_r}\log^2\abs{L_n(z)}\,d\lambda(z) &\leq \int_{\D_r}\Bigg(\frac{3}{n^2}I_n^2(z;R) + \frac{3t_n}{n^2}\sum_{t=1}^{t_n}\log^2\abs{\frac{R\left(z-w_{t}^{(n)}\right)}{R^2 - \overline{w_{t}^{(n)}}\,z}} \\
&\qquad\qquad + \frac{3s_n}{n^2}\sum_{s=1}^{s_n}\log^2\abs{\frac{R\left(z-y_s^{(n)}\right)}{R^2 - \overline{y_s^{(n)}}\,z}}\Bigg)\,d\lambda(z)
\end{aligned}
\label{eqn:kabl:finalIneq2}
\end{gather}
 almost surely. Lemmas \ref{lma:kabl:bdpint} and \ref{lma:kabl:bdpintbelow} establish that 
\begin{equation*}
\lim_{n\to \infty}\P\left(\abs{\frac{3}{n^2}\int_{\D_r}I^2_n(z;R)\,d\lambda(z)} \geq C\right) = 0
\end{equation*}
for some constant $C > 0$, and hence the sequence $\frac{3}{n^2}\int_{\D_r}I^2_n(z;R)\,d\lambda(z)$ is tight. 

The remaining two terms of \eqref{eqn:kabl:finalIneq2} are bounded almost surely. Indeed, for $z \in \D_r$ and $y_s^{(n)} \in \D_R$, we have 
\[
\frac{|z-y_s^{(n)}|}{2R} \leq \abs{\frac{R(z-y_s^{(n)})}{R^2- \overline{y_s^{(n)}}\,z}}\leq \frac{|z-y_s^{(n)}|}{R-r},
\]
and hence 
\[
\log^2{\abs{\frac{R(z-y_s^{(n)})}{R^2- \overline{y_s^{(n)}}\,z}}} \leq \log^2{\frac{|z-y_s^{(n)}|}{2R}} + \log^2{\frac{|z-y_s^{(n)}|}{R-r}}. 
\]
By a simple change of variables, we obtain
\[
\int_{\D_r}\log^2{\frac{|z-y_s^{(n)}|}{2R}}\,d\lambda(z) \leq \int_{\D_{2R}}\log^2{\frac{|z|}{2R}}\,d\lambda(z),
\]
and similarly
\[
\int_{\D_r}\log^2{\frac{|z-y_s^{(n)}|}{R-r}}\,d\lambda(z) \leq \int_{\D_{2R}}\log^2{\frac{|z|}{R-r}}\,d\lambda(z).
\]
Thus, by the local integrability of the squared logarithm, 
\[
\frac{3s_n}{n^2}\int_{\D_r}\sum_{s=1}^{s_n}\log^2\abs{\frac{R\left(z-y_s^{(n)}\right)}{R^2 - \overline{y_s^{(n)}}\,z}}\,d\lambda(z) \leq \frac{3s_n^2}{n^2}C' \leq 3C'
\]
almost surely for all $n \in {\mathbb N}$, where $C' > 0$ is a constant that depends only on $R$ and $r$, and, in the last inequality, we used the fact that $s_n \leq n$.  A similar argument applies to the integral of the sum in \eqref{eqn:kabl:finalIneq2} involving the critical points $w_t^{(n)}$; we omit the details.  

We conclude that the sequence $\frac{1}{n^2}\int_{\D_r}\log^2\abs{L_n(z)}\,d\lambda(z)$ is tight, and the proof of Lemma \ref{lma:kabl:intTight} is complete.

\section*{Acknowledgement} The authors would like to thank Boris Hanin for providing useful comments and suggestions on an earlier version of the manuscript.

\end{document}